\documentclass[11pt]{amsart}

\usepackage{amssymb}
\usepackage{lscape}
\usepackage{mathpazo}
\usepackage{mathrsfs}
\usepackage{graphicx}
\usepackage[all]{xy}
\usepackage[headings]{fullpage}
\usepackage{paralist}
\usepackage{tikz,tikz-3dplot}
\usepackage{float}
\usepackage{tabularx,ragged2e,booktabs,caption}
\usepackage[colorlinks=true,linkcolor=blue, citecolor=blue, urlcolor=blue,  pagebackref=true]{hyperref}
\usepackage[normalem]{ulem}

\makeatletter

\newcounter{maintheorem}[equation]
\def\themaintheorem{\thesection.\@arabic \c@maintheorem}
\@addtoreset{equation}{section}
\def\theequation{\thesection.\@arabic \c@equation}

\theoremstyle{plain}
\newtheorem{theorem}[equation]{Theorem}
\newtheorem*{theorem*}{Theorem}
\newtheorem{lemma}[equation]{Lemma}

\newtheorem{proposition}[equation]{Proposition}

\theoremstyle{definition}

\newtheorem{fact}[equation]{Fact}

\newtheorem{remark}[equation]{Remark}

\newtheorem{example}[equation]{Example}

\newtheorem{definition}[equation]{Definition}

\newtheorem{notation}[equation]{Notation}

\newtheorem{discussion}[equation]{Discussion}

\newtheorem{observation}[equation]{Observation}

\newtheorem{construction}[equation]{Construction}

\DeclareMathOperator{\ann}{Ann}

\DeclareMathOperator{\core}{core}
\DeclareMathOperator{\D}{\mathcal D}

\DeclareMathOperator{\Hom}{Hom}

\DeclareMathOperator{\HS}{HS}

\DeclareMathOperator{\M}{\mathcal M}

\DeclareMathOperator{\reg}{reg}

\DeclareMathOperator{\socdeg}{socdeg}
\DeclareMathOperator{\socle}{Soc}

\DeclareMathOperator{\type}{type}

\def\theenumi{\@alph\c@enumi}
\def\theenumii{\@roman\c@enumii}

\makeatother

\def\to{\longrightarrow}

\def\RDerChar{\mathbf{R}}
\def\RDer{\@ifnextchar[{\R@Der}{\ensuremath{\RDerChar}}}
\def\R@Der[#1]{\ensuremath{\RDerChar^{#1}}}
\makeatother

\def\cfudot#1{\ifmmode\setbox7\hbox{$\accent"5E#1$}\else
	\setbox7\hbox{\accent"5E#1}\penalty 10000\relax\fi\raise 1\ht7
	\hbox{\raise.1ex\hbox to 1\wd7{\hss.\hss}}\penalty 10000 \hskip-1\wd7\penalty
	10000\box7}

\makeatletter
\newcommand*\bigcdot{\mathpalette\bigcdot@{.5}}
\newcommand*\bigcdot@[2]{\mathbin{\vcenter{\hbox{\scalebox{#2}{$\m@th#1\bullet$}}}}}
\makeatother

\newcommand{\N}{\mathbb{N}}

\newcommand{\m}{{\mathfrak m}}

\newcommand{\ee}{\mathbf{e}}
\newcommand{\nn}{\mathbf{n}}
\newcommand{\kk}{\mathbf{k}}

\newcommand{\ua}{\underline{\textit{a}}}

\newcommand{\uz}{\underline{z}}
\newcommand{\ux}{\underline{x}}

\makeatletter

\newcommand{\etalchar}[1]{$^{#1}$}

\title[Inverse System of Level $K$-Algebras]
{The Structure of the Inverse System of Level $K$-Algebras}

\author[Masuti] {Shreedevi K. Masuti}
\address{Dipartimento di Matematica, Universit{\`a} di Genova, Via Dodecaneso 35, 16146 Genova, Italy}
\email{masuti@dima.unige.it}

\author[Tozzo]{Laura Tozzo}
\address{Department of Mathematics, University of Kaiserslautern, 67663 Kaiserslautern, Germany}
\curraddr{Dipartimento di Matematica, Universit{\`a} di Genova, Via Dodecaneso 35, 16146 Genova, Italy}
\email{tozzo@dima.unige.it}

\date{\today}
\thanks{The first author was supported by INdAM COFOUND Fellowships cofounded by Marie Curie actions, Italy.}

\keywords{Macaulay's inverse system, level $K$-algebras, Gorenstein $K$-algebras}
\subjclass[2010]{Primary: 13A02, 13H10; Secondary: 13J05, 13F20}

\begin{document}

\begin{abstract}

Macaulay's inverse system is an effective method to construct Artinian $K$-algebras with additional properties like, Gorenstein, level, more generally with any socle type.   
Recently, Elias and Rossi \cite{ER17} gave the structure of the inverse system of $d$-dimensional Gorenstein $K$-algebras for any $d>0$. In this paper we extend their result by establishing a one-to-one correspondence between $d$-dimensional level $K$-algebras and certain submodules of the divided power ring. We give several examples to illustrate our result. 
\end{abstract}

\maketitle

\section{Introduction}

Level rings have been studied in several different contexts. These rings are in between Cohen-Macaulay and Gorenstein rings. Their study was initiated by Stanley in \cite{Stanley77}. 
Since then they have been widely investigated, especially in the Artinian case, see for instance \cite{Boij},\cite{Boi99}, \cite{Boi00}, \cite{Boi09}, \cite{Bertella}, \cite{GHMS}, \cite{Iar84}, \cite{Stefani}. However, there are also many examples of level rings in positive dimension: Stanley-Reisner rings of matroid simplicial complexes \cite{Stanley77}, associated graded rings of semigroup rings corresponding to arithmetic sequences \cite{MT95}, determinantal rings corresponding to generic matrices \cite{BV88} or generic symmetric matrices \cite{Con94}, \cite{Goto}. 

\vskip 2mm
Although several theory has been developed for level rings, they are not as well-understood as 
Gorenstein rings. One of the reasons for the lack of knowledge of level $K$-algebras of positive dimension is the absence of   an effective method to construct level $K$-algebras. Macaulay's inverse system (see \cite{Mac1916}) allows to construct level rings in the zero-dimensional case, as was shown by Emsalem \cite{Ems78} and Iarrobino \cite{Iar94}.
Recently, Elias and Rossi \cite{ER17} gave the structure of the inverse system of  Gorenstein $K$-algebras of any dimension $d>0.$ In this paper we extend their result and give the structure of the inverse system of level $K$-algebras of positive dimension. 

\vskip 2mm
We define local level rings of positive dimension and study their properties in Section \ref{Section:LevelAlgebras}. The definition of level rings is well-known in the graded case.  
Recall that a homogeneous $K$-algebra $A$ is \emph{level} if the canonical module $\omega_A$ of $A$ has a minimal set of generators of same degree (see Definition \ref{Def:GradedLevel}).
For the local case we take inspiration from \cite{EI87}, and say that a local $K$-algebra $A$ is \emph{level} if $A/J$ is Artinian level for some general minimal reduction $J$ of the maximal ideal (see Definition \ref{Def:LocalLevel}).
Notice that defining local level $K$-algebras of positive dimension is non-trivial. One of the reasons is that, if the associated graded ring of $A$ is not Cohen-Macaulay, then Artinian reductions of $A$ by minimal reductions of the maximal ideal may have different socle types (see Example \ref{Example:NotLocalLevel} and Proposition \ref{Prop:construction of level}). 

\vskip 2mm
In Section \ref{Section:DividedRing} we recall some classical results about divided power rings and the inverse system.
Let $R=K[x_1,\ldots,x_m]$ be a standard graded polynomial ring or $R=K[\![x_1,\ldots,x_m]\!]$ a power series in $m$ variables over the field $K$, and let $I$ be an ideal of $R$ (homogeneous if $R=K[x_1,\ldots,x_m]$).  
It is well-known that the the injective hull of $K$ is isomorphic to the \emph{divided power ring} $\D:=K_{DP}[X_1,\ldots,X_m]$  as an $R$-module, where $\D$ has a structure of $R$-module via the contraction action (see Section \ref{Section:DividedRing}). 
Therefore, the dual module $(R/I)^\vee:=\Hom_R(R/I,E_R(K))=\Hom_R(R/I,\D)$ is isomorphic to an $R$-submodule of $\D$, called the \emph{inverse system of $I$} and denoted by  $I^\perp$. 
By Matlis duality it is clear that if $R/I$ has positive Krull dimension, then  $I^\perp$ is not finitely generated. 

\vskip 2mm
In Section \ref{Section:MainResult} we investigate the structure of $I^\perp$ when $R/I$ is a positive dimensional level $K$-algebra. 
In \cite{ER17} the authors proved that $I^\perp$ 
is \emph{$G_d$-admissible} if $R/I$ is Gorenstein of dimension $d$ and, vice versa, for any $G_d$-admissible submodule $W$ of $\D,$ $W^\vee$ is a $d$-dimensional Gorenstein $K$-algebra. 
Generalizing this result, in Theorem \ref{thm:CharOfLocalLevel}, we establish a one-to-one correspondence between level $K$-algebras $R/I$ of dimension $d>0$ and \emph{$L_d^\tau$-admissible} submodules of $\D$ (see Definition \ref{Def:LdAdmissible}). 
Observe that the conditions given for $L_d^\tau$-admissibility are not merely the ``union'' of the conditions given for $G_d$-admissible submodules (see Discussion \ref{Disc:LdAdmissibleDefinition}). 
This is a symptom of the intrinsic complexity of level $K$-algebras in positive dimension, and it is one of the reasons why constructing examples of level algebras is hard.
For example, in the Artinian case, as a consequence of Macaulay's inverse system, it is known that the intersection of Gorenstein ideals of same socle degree is always level. 
The analogous is not true in positive dimension (see Example \ref{Example:intersection}). 
For this reason Theorem \ref{thm:CharOfLocalLevel} is an important tool, as it gives an effective method to construct level $K$-algebras.

\vskip 2mm
In Section \ref{Section:Examples}, we construct several examples of level $K$-algebras.

We have used the computer algebra systems \cite{Macaulay2}, \cite{Singular} and the library \cite{Elias} 
for various computations in this paper.

\section{Level $K$-algebras}
\label{Section:LevelAlgebras}

Throughout this paper let $R=K[x_1,\ldots,x_m]$ be a standard graded polynomial ring or $R=K[\![x_1,\ldots,x_m]\!]$ a power series ring in $m$ variables over an infinite field $K$, let $\M=(x_1,\ldots,x_m)$ be the unique maximal (homogeneous) ideal and $I$ an ideal of $R.$ We write $A=R/I,$ $\m=\M/I$, and assume that $A$ is Cohen-Macaulay of dimension $d>0.$ If $R=K[x_1,\ldots,x_m]$ and $I$ is a homogeneous ideal in $R$, we say that $A$ is {graded} (or {homogeneous}). 
We write $\socle(A):=(0:_A\m)$ for the {socle} of $A$ and set $\tau:=\mathrm{type}(A).$

In this section we introduce the concept of local level $K$-algebras and derive some properties of these algebras that we need to construct the inverse system. In the literature 
level $K$-algebras have been defined for zero-dimensional rings in the local case and for rings of arbitrary dimension in the graded case. 
We define local level $K$-algebras of positive dimension using general minimal reductions (Definition \ref{Def:LocalLevel}). 
This definition, in particular, ensures that if $A$ is level and $(\ua)$ is a general minimal reduction of $\m,$ then  $A/(\ua^\nn)$ is level for every $\nn \in \N^d_+$ (Proposition \ref{Prop:QuotientLevel}).  
This is a necessary property for a ring to be level as we will see in Section \ref{Section:MainResult}.

\begin{definition}\label{Def:GradedLevel}
A homogeneous Cohen-Macaulay $K$-algebra $A$ is called {\it level} if the canonical module of $A$, $\omega_A$, has a minimal set of generators of same degree.  
\end{definition}
\noindent Equivalently, there exists an integer $c$ such that $\beta_{n-d,j}^R(A) = 0$ if and only if $j \neq c$, where $\beta_{i,j}:=\beta_{i,j}^R(A)$ denote the graded Betti numbers of $A.$ In other words, a minimal $R$-free resolution of $A$ is of the form:
\[
0 \to R(-c)^{\beta_{n-d,c}} \to \cdots \to \bigoplus_{j} R(-j)^{\beta_{1,j}} \to R \to 0.
\]

In the following proposition we recall the well-known fact that a homogeneous $K$-algebra $A$ is level if and only if $A/(a)$ is level for any homogeneous $A$-regular element $a$ in $\m.$ 
 
\begin{proposition}
\label{Prop:GLQutotientByHomEle}
 Let $A=R/I$ be a homogeneous $K$-algebra and let $a \in A$ be a homogeneous $A$-regular element. Then $A$ is level if and only if $A/({a})$ is level. 
\end{proposition}
\begin{proof}
 For a graded $A$-module $M$ and an integer $b,$ let $M(b)$ denote the $A$-module $M$ with grading given by $M(b)_i=M_{b+i}.$ Let $b:=\deg(a).$ By \cite[Corollary 3.6.14]{BH98} 
\[
\omega_{A/(a)} \simeq (\omega_A/(a)\omega_A)\left(b\right).
\]
Hence it follows that $A$ is level if and only if $A/(a)$ is level.
\end{proof}

From Definition \ref{Def:GradedLevel} it is clear that we can not define level $K$-algebras analogously in the local case. However, the definition of Artinian local level $K$-algebras is quite well-known. Let us recall 
this definition. The \emph{socle degree} of an Artinian local ring, denoted as $\socdeg(A),$ is the maximum integer $j$ such that $\m^j \neq 0.$ An Artinian local $K$-algebra $A=R/I$ is said to be \emph{level} if 
$\socle(A)=\m^s$ where $s:=\socdeg(A).$

Let $(A,\m)$ a local ring of positive dimension. 
Recall that an ideal $J \subseteq \m$ is said to be a \emph{reduction} of $\m$ if there exists a non-negative integer $k$ such that $\m^{k+1}=J \m^k.$ If further $J$ does not contain properly any other reductions of $\m, $ then we say that $J$ is a \emph{minimal reduction} of $\m.$ 
We refer to \cite[Section 1.2]{RV10} and \cite[Chapter 8]{HS06} for the basic properties of reductions and superficial elements. 
The following example illustrates that in the non-graded case it is possible that $A/J$ is level for some  minimal reduction $J$ of $\m$, but $A/J'$ is not level for another minimal reduction $J'\ne J.$ This is one of the reasons why defining local level $K$-algebras of positive dimension 
is not trivial. 

\begin{example}
\label{Example:NotLocalLevel}
Let $A=\mathbb{Q}[\![t^6,t^7,t^{11},t^{15}]\!]$.
Then $A$ has type $2$ and $J=(t^6)$ and $J'=(t^6+t^7)$ are two minimal reductions of $\m=(t^6,t^7,t^{11},t^{15})$. Using a computer algebra system, it can be verified that $A/J$ has the Hilbert function $(1,3,2)$ and hence is level. But $A/J'$ has the Hilbert function $(1,3,1,1)$ and hence is not level. Notice that here both $t^6$ and $t^6+t^7$ are superficial elements for $\m.$
\end{example}

We define level $K$-algebras using general minimal reductions. General minimal reductions have been introduced by several authors. Here we recall the definition due to P.~Mantero and Y.~Xie \cite{MX16}.

\begin{definition}\label{Def:genelem}
Let $a_i=\sum_{j=1}^m a_{ij}(\overline{x_j}) \in \m$ 
for $i=1,\ldots,t$ where $(a_{ij}) \in A^{tm}.$ 
We say that the  elements $ a_1,\ldots, a_t$ are {\it general} if there exists a proper Zariski-dense open subset $U$ of $(A/\m)^{tm}=K^{tm}$ such that $(\overline{a_{ij}})\in U$ where $\overline{(\cdot)}$ denotes image in $K.$
\end{definition}

We recall the following fact on general elements:
\begin{fact}
\label{Fact:GeneralElements}
Since $K$ is infinite, a sequence $\ua:=a_1,\ldots,a_d $ of general elements in $\m$ always forms a superficial sequence for $\m$ \cite[Corollary 2.5]{Xie12}. In particular, if $a_1,\ldots,a_d $ are general elements in $\m,$ then $(\ua)$ is a minimal reduction of $\m .$
\end{fact}

\begin{definition}\label{Def:genminred}
We say that $J=(a_1,\ldots,a_d) \subseteq \m$ is a {\it general minimal reduction} of $\m$ if $a_1,\ldots,a_d$ forms a sequence of general elements in $\m.$ 
\end{definition}

We recall the following lemma from \cite{MX16} which guarantees that the socle degree and the Hilbert function of an Artinian reduction of $A$ are the same for every general minimal reduction of $\m$:
\begin{lemma} \cite[Proposition 3.2]{MX16}
\label{Lemma:IndMinRed}
Let $i\geq 0$.
Then the socle degree of $A/J$ and $\dim_K (\m^i+J/J)$ are independent of the general minimal reduction $J$ of $\m$.
\end{lemma}

In \cite[\S 1]{EI87} authors defined general minimal reductions in a different way. 
It is not clear whether their definition and Definition \ref{Def:genminred} are the same. More generally,  they proved that the socle type of $A/J$ is the same for every general minimal reduction $J$ of $\m$ and define the socle type of $A$ as the socle type of $A/J$ for a general minimal reduction $J$ of $\m$ (see \cite{Iar84} for the definition of socle type in the Artinian case). 
Motivated by this, we define level local $K$-algebras as follows:

\begin{definition}\label{Def:LocalLevel}
A Cohen--Macaulay $K$-algebra $A=R/I$, where $R=K[\![x_1,\ldots,x_m]\!]$, is said to be {\it level} if $A/J$ is an Artinian level $K$-algebra for some (equiv. every) general minimal reduction $J$ of $\m.$ 
\end{definition}

Clearly, both in the graded and in the local case, if $A$ is Gorenstein, then $A$ is level.

We point out that we can not remove ``general'' in the definition of level local $K$-algebra. 
In fact, in Example \ref{Example:NotLocalLevel} $A/(t^6)$ is level, but $A$ is not level. Notice that in this example $(t^6)$ is a minimal reduction of $\m,$ but $(t^6)$ is not a general minimal reduction of $\m.$ One of the main reasons why the property being $A/J$ level depends on the minimal reduction $J$ is that the associated graded ring $gr_\m(A):=\oplus_{i \geq 0}\m^i/\m^{i+1}$ of a Cohen-Macaulay local ring $A$ need not be Cohen-Macaulay. 
The following lemma shows that if $gr_\m(A)$ is Cohen-Macaulay and $A/J$ is level for some minimal reduction, then $A/J'$ is level for every minimal reduction $J'$ of $\m$. This gives an effective method to examine whether a given local ring is level. 

\begin{proposition}\label{Prop:construction of level}
Assume that $G:=gr_\m(A)$ is Cohen-Macaulay.  
Then the following assertions hold true:
\begin{enumerate}
 \item The socle degree of $A/J$ is the same for every minimal reduction $J$ of $\m.$
\item If $A/J$ is level for some minimal reduction $J$ of $\m$, then $A/J'$ is level for every minimal reduction $J'$ of $\m.$ In particular, if $A$ is level, then $A/J$ is level for every minimal reduction $J$ of $\m.$
\end{enumerate}
\end{proposition}
\begin{proof}
Let $J \subseteq \m$ be a minimal reduction of $\m$. 
Then there exist $\ua:=a_1,\ldots,a_d \in \m$ such that $J=(\ua)$ and $\ua$ is a superficial sequence for $\m.$ Since $G$ is Cohen-Macaulay, $a_1^*,\ldots,a_d^*$ is $G$-regular. For a graded ring $S,$ let $\HS_S(t)$ denote the Hilbert series of $S.$ Hence if $HS_G(t)=\frac{1+h_1t+\cdots+h_s t^s}{(1-t)^{d}}$ with $h_s \neq 0$ and $\sum_{i=1}^s h_i\ne 0$, then
\[
HS_{gr_{\m/J}(A/J)}(t)=1+h_1t+\cdots+h_s t^s .
\]
This implies that $s$ is the socle degree of $A/J$ and $h_s=\dim_K (\m^s+J/J).$ This proves that the socle degree and $\dim_K (\m^s+J/J)$ are the same for every minimal reduction $J$ of $\m.$ Hence the result follows.
\end{proof}

However, if $G$ is not Cohen-Macaulay, then it is not clear whether $A/J$ level for a general minimal reduction implies that $A/J'$ is level for every minimal reduction $J'$ of $\m.$

\begin{example} \cite[Example 3, p. 125]{RV00}
\label{Example:Rossi-Valla}
Consider the semigroup ring 
\[
A=K[\![t^6,t^8,t^{10},t^{13}]\!]\cong K[\![x,y,z,w]\!]/(y^2-xz,yz-x^3,z^2-x^2y,w^2-x^3y).
 \]
 Then $A$ is Cohen-Macaulay of type $2$ and $gr_\m(A)$ is Cohen-Macaulay. 
Moreover, $A/(t^6)$ has the Hilbert function $(1,3,2)$ and hence is level. Therefore by Proposition \ref{Prop:construction of level} $A$ is level. 
\end{example}

Let $(\ua),$ where $\ua:=a_1,\ldots,a_d$ (is a regular linear sequence if $A$ is homogeneous), be a minimal reduction of $\m$. For $\nn \in \N^d_+$ we set $\ua^\nn=a_1^{n_1},\ldots,a_d^{n_d}.$ 
One of the important properties that a graded ring satisfies is: if $A/(\ua)$ is level (equiv. $A$ is level), then  $A/(\ua^\nn)$ is level for all $\nn \in \N^d_+$.  
This is no longer true in the local case. In fact, in Example \ref{Example:NotLocalLevel} $A/(t^6)$ is level, but $A/(t^6)^2$ is not level. 
However, in Proposition \ref{Prop:QuotientLevel}\eqref{Prop:QuotientLevelLocal} we will show that if $(\ua)$ is a general minimal reduction of $\m$ and $A/(\ua)$ is level (equiv. $A$ is level) with $\mathrm{char}(K)=0$, then $A/(\ua^\nn)$ is level for all $\nn \in \N^d_+.$ 
This is another reason why we need ``general'' in Definition \ref{Def:LocalLevel}. 
We recall few definitions and results needed to prove this. 

Following \cite{MX16} we define the {\it index of nilpotency of $A$ with respect to a reduction $J$ of $\m$} as
\[
 s_J(A):=\socdeg A/J.
\]

\begin{proposition}\cite[Proposition 3.2]{MX16} \cite[Proposition 5.3.3]{Fouli}
\label{Prop:IndOfMult}
If $J$ is a general minimal reduction of $\m,$ then $s_J(A)$ does not depend on a choice of $J.$ Moreover, if $R$ is equicharacteristic and $J$ is a general minimal reduction of $\m$, then 
$$s_{J'}(A) \leq s_J(A)$$
for every minimal reduction $J'$ of $\m.$
\end{proposition}

In \cite[Page 346]{EI87} authors obtained an analogous result which does not depend upon the characteristic of $R.$ Since it is not clear whether the definition of general elements given in \cite{EI87} and Definition \ref{Def:genelem} are the same,  it is not known whether  Proposition \ref{Prop:IndOfMult} is true if $R$ is not equicharacteristic.   

\vskip 2mm
Following \cite{MX16} we call the number
\[
s(A):=s_J(A) \mbox{ where $J$ is a general minimal reduction of $\m$},
\]
the {\it index of nilpotency of $A$} which is well-defined by Proposition \ref{Prop:IndOfMult}. 

We need the following theorem from \cite{HT05} on the core of an ideal. Recall that the \emph{core} of an ideal $I$ is the intersection of all (minimal) reductions of $I.$ 
The core is difficult to compute, and there have been many efforts to find explicit formulas, see  \cite{HT05}, \cite{PU05} and references therein.
We recall  the following theorem which gives an explicit formula to compute the core. The theorem is true more generally for equimultiple ideals in any Cohen-Macaulay local ring. Since we need this for $\m,$ we state the result only for $\m.$ Recall that for a reduction $J$ of $\m$, the \emph{reduction number of $\m$ with respect to $J$}, denoted as 
$r_J(\m)$, is the smallest nonnegative integer $k$ such that the equality $\m^{k+1}=J\m^k$ holds true. 
The \emph{reduction number of $\m$} is 
\[
r(\m):=\min\{r_J(\m)\mid J \text{ is a minimal reduction of }\m\}.
\]

\begin{theorem} \cite[Theorem 3.7]{HT05}
\label{thm:HunekeTrung}
Let $R=K[\![x_1,\dots,x_m]\!]$ with $\mathrm{char}(K)=0$, and let $A=R/I$ be Cohen--Macaulay of dimension $d>0$.
Let $J$ be a minimal reduction of $\m$ and $r:=r(\m)$ the reduction number of $\m$. Then
\[
\core(\m)= J^{r+1}:\m^r.
\]
Equivalently, $\core(\m)= J^{n+1}:\m^n$ for all $ n \geq r.$
\end{theorem}

We remark that Theorem \ref{thm:HunekeTrung} is not true in positive characteristic (see \cite[Example 4.9]{PU05}). Therefore from now onwards we assume that the characteristic of $K$ is zero in the case 
$R=K[\![x_1,\ldots,x_n]\!].$ 

We use the following notations:
\begin{align*}
|\nn|&:=n_1+\cdots+n_d \hspace*{0.5cm} \mbox{ for } \nn \in \N^d_+\\
\ee_i &:=(0,\ldots,1,\ldots,0) \in \N^d, \\
{\bf t}_d&:=(t,\ldots,t) \in \N^d \hspace*{0.5cm} \mbox{ for } t \in \N. 
\end{align*}

In the following proposition we prove that $A/(\ua^\nn)$ has ``expected'' socle degree. This is an important result that we need in constructing the inverse system of level  algebras.  
We thank A. De Stefani for providing a proof of this (through private communication) in the one-dimensional case. The same idea works for higher dimension. 
\begin{proposition}
\label{Prop:SocDegOfR_nn}
Let $A=R/I$ and $\ua:=a_1,\ldots,a_d \in \m$ be one of the following: 
\begin{enumerate}
 \item \label{Prop:SocDegOfR_nnGraded} $ R=K[x_1,\ldots,x_m],$ $\ua$ is a regular linear sequence in $\m$ and $s$ the Castelnuovo-Mumford regularity of $A;$ 
 \item \label{Prop:SocDegOfR_nnLocal} $R=K[\![x_1,\ldots,x_m]\!]$ with $\mathrm{char}(K)=0,$ $(\ua)$ a general minimal reduction of $\m$ and $s$ the index of nilpotency of $A.$  
\end{enumerate}
Then for all $\nn \in \N^d_+$,
\begin{equation}\label{Eqn:socledegree}
\socdeg(A/(\ua^\nn))=s+|\nn|-d.
\end{equation}
\end{proposition}
\begin{proof}\pushQED{\qed}
\begin{asparaenum}[(a)]
\item Suppose $R=K[x_1,\ldots,x_m] $ and $\ua \in \m$ is a regular linear sequence in $A.$ 
We prove \eqref{Eqn:socledegree} by induction on $|\nn|$. Notice that the Castelnuovo-Mumford regularity of an Artinian graded ring coincides with its socle degree. Since 
the Castelnuovo-Mumford regularity of $A,$ $\reg(A),$ and regularity of the Artinian reduction $A/(\ua)A$ are the same, we have 
\[
 \socdeg(A/(\ua))=\reg(A/(\ua))=\reg(A)=s.
\]
Hence the assertion is clear for $|\nn|=d$. Let $|\nn| > d$. After a permutation, we may assume that $\nn=(n_1,\ldots,n_d)$ with $n_1 \ge 2$. By induction $A/(\ua^{\nn-\ee_1})$ has socle degree $s+|\nn|-d-1$. Let $f \in \m^{s+|\nn|-d+1}$ be a homogeneous polynomial. 
Since $\m^{s+|\nn|-d+1} \subseteq (\ua^{\nn-\ee_1})$,
\[
f=a_1^{n_1-1}f_1+a_2^{n_2}f_2+\cdots+a_d^{n_d}f_d 
\]
where $f_i \in A$ are homogeneous polynomials. 
Thus $\deg(f_1) \geq  s+|\nn|-d+1 -(n_1-1).$ By induction hypothesis applied to $(1,n_2,\ldots,n_d),$ we get that 
\[f_1 \in \m^{s+|\nn|-(n_1-1)-d+1} \subseteq (a_1,a_2^{n_2},\ldots,a_d^{n_d}).\]
Therefore $f \in (\ua^\nn)$. 
This gives that $\m^{s+|\nn|-d+1} \subseteq (\ua^\nn)$. 

On the other hand, if $f\in \m^{s+|\nn|-d-1} \setminus (\ua^{\nn-\ee_1})$ is a homogeneous polynomial,  
then $a_1 f \in  \m^{s+|\nn|-d} \setminus (\ua^\nn)$. 
Hence $\m^{s+|\nn|-d} \nsubseteq (\ua^\nn)$. This proves \eqref{Eqn:socledegree}.
\vskip 3mm

\item Let $R=K[\![x_1,\ldots,x_m]\!]$ with $\mathrm{char}(K)=0$ and $({\ua})$ a general minimal reduction of $\m.$ 
Let $V$ be the set of all minimal reductions of $\m.$ 
As $\mathrm{char}(K)=0,$  $R$ is equicharacteristic. Hence by Proposition \ref{Prop:IndOfMult}, $\socdeg (A/J) \leq s(A)=s$ for any minimal reduction $J$ of $\m.$ 
Therefore
\[
\m^{s+1}\subseteq \bigcap_{J\in V}J=\core(\m).
\]
We prove \eqref{Eqn:socledegree} by induction on $|\nn|.$ Let $J:=(\ua).$
Since $s$ is the socle degree of $A/(\ua),$ the assertion is clear if $|\nn|=d.$ Let $|\nn|\geq d+1$. Assume that $k:=\max\{0,r-|\nn|+d\}$ where $r:=r(\m)$ is the reduction number of $\m.$ 
Since $k+|\nn|-d\ge r,$ by Theorem \ref{thm:HunekeTrung} $\core(\m)=J^{k+|\nn|-d+1}:\m^{k+|\nn|-d}.$ Therefore
\begin{align}
\label{Eqn:ContainmentForallk}
J^k\m^{s+|\nn|-d+1}&\subseteq \m^{k+|\nn|-d}\m^{s+1}\subseteq \m^{k+|\nn|-d}\core(\m) 
\subseteq J^{k+|\nn|-d+1}. 
\end{align}
Note that $J^{|\nn|-d+1}\subseteq (\ua^{\nn}).$ Indeed suppose that $a_1^{n_1'}\cdots a_d^{n_d'} \in J^{|\nn|-d+1}$ where $\nn'=(n'_1,\ldots,n'_d)\in \N^d$ and $|\nn'|\ge |\nn|-d+1.$ Then $n'_i\ge n_i$ for some $i\in\{1,\dots,d\}$ and hence $a_1^{n_1'}\cdots a_d^{n_d'}\subseteq (\ua^{\nn}).$ Therefore if $k=0,$ then by \eqref{Eqn:ContainmentForallk}
\[
\m^{s+|\nn|-d+1} \subseteq J^{|\nn|-d+1} \subseteq (\ua^{\nn}).
\]
Assume that $k \geq 1.$ Since $\oplus_{i \geq 0} J^i/J^{i+1}$ is Cohen-Macaulay, $J^{i+1}:(a_1)=J^{i}$ for all $i \geq 0.$ 
Hence by \eqref{Eqn:ContainmentForallk}
\[
\m^{s+|\nn|-d+1}\subseteq J^{k+|\nn|-d+1}:J^k \subseteq J^{k+|\nn|-d+1}:(a_1^k)= J^{|\nn|-d+1}\subseteq(\ua^{\nn}).
\]
Thus $\socdeg(A/(\ua^\nn))\le s+|\nn|-d.$

On the other hand, if $u\in \m^{s}\setminus (\ua)$, then $(a_1^{n_1-1}\cdots  a_d^{n_d-1})u\in \m^{s+|\nn|-d}\setminus (\ua^\nn)$. Thus $\m^{s+|\nn|-d} \nsubseteq (\ua^\nn).$ Hence $\socdeg (A/(\ua^\nn)) =s+|\nn|-d.$ \qedhere
\end{asparaenum}
\end{proof} 

\begin{proposition}\label{Prop:QuotientLevel}
\begin{enumerate}
\item \label{Prop:QuotientLevelGraded}
Let $R=K[x_1,\dots,x_m]$ and ${\ua}:={a_1},\ldots,{a_d}$ be a regular linear sequence in $\m$. 
If $A$ is graded level of type $\tau$, then  $A/(\ua^\nn)$ is an Artinian graded level $K$-algebra of type $\tau$ and socle degree $s+|\nn|-d$ for all $\nn \in \N^d_+$ where $s$ is the Castelnuovo-Mumford regularity of $A.$

\item \label{Prop:QuotientLevelLocal}
Let $R=K[\![ x_1,\dots, x_m ]\!]$ with $\mathrm{char}(K)=0$ and $({\ua}) \subseteq \m$ be a general minimal reduction of $\m$ where $\ua:=a_1,\ldots,a_d.$ If $A$ is local level ring of type $\tau$, then  $A/(\ua^\nn)$ is an Artinian local level $K$-algebra of type $\tau$ and socle degree $s+|\nn|-d$ for all $\nn \in \N^d_+$ where $s$ is the index of nilpotency of $A.$
\end{enumerate}
\end{proposition}
\begin{proof}\pushQED{\qed}
\begin{asparaenum}
\item 
Follows from Propositions \ref{Prop:GLQutotientByHomEle} and \ref{Prop:SocDegOfR_nn}.
\item 
Since $A$ is level, by definition $A/(\ua)$ is an Artinian local level $K$-algebra of type $\tau.$ 
Let $s_{\nn}:=\socdeg(A/(\ua^\nn)).$ 
By Proposition \ref{Prop:SocDegOfR_nn}\eqref{Prop:SocDegOfR_nnLocal}, $s_\nn=s+|\nn|-d.$ Consider
\[
\mu: \frac{\m^s+(\ua)}{(\ua)} \xrightarrow{\cdot a_1^{n_1-1}\cdots a_d^{n_d-1}} \frac{\m^{s_\nn}+(\ua^\nn)}{(\ua^\nn)}.
\]
Clearly $\mu$ is well-defined. Since $\ua$ is a regular sequence in $A,$ it is easy to verify that $\mu$ is injective. Hence 
\[
\dim_K \frac{\m^{s_\nn} + (\ua^\nn)}{(\ua^\nn)} \geq \dim_K \frac{\m^s+(\ua)}{(\ua)}=\dim_K \socle(A/(\ua))=\tau.
\]
On the other hand, one always has 
\[
\dim_K  \frac{\m^{s_\nn} + (\ua^\nn)}{(\ua^\nn)} \leq \tau
\]
as $A/(\ua^\nn)$ is an Artinian local ring of type $\tau$ and $(\m^{s_\nn} + (\ua^\nn))/(\ua^\nn) \subseteq \socle(A/(\ua^\nn)). $ 
Therefore $\dim_ K (\m^{s_\nn} + (\ua^\nn))/(\ua^\nn)=\tau$ and hence $(\m^{s_\nn} + (\ua^\nn))/(\ua^\nn)=\socle(A/(\ua^\nn)).$ Thus $A/(\ua^\nn)$ is level.\qedhere
\end{asparaenum}
\end{proof}

\section{Divided Power rings and Inverse System}
\label{Section:DividedRing}
In this section we recall some basic facts about the divided power rings and Macaulay's inverse system for Artinian level $K$-algebras. These play an important role in constructing the inverse system of level $K$-algebras of dimension $d>0$ in the next section. 

Let $R=K[\![x_1,\ldots,x_m]\!]$ or $R=K[x_1,\ldots,x_m]$ and let $R_i$ denote the $K$-vector space spanned by homogeneous 
polynomials of degree $i.$ 
It is known that the injective hull of $K$ as an $R$-module is isomorphic to the following ring (see \cite[Exercise A3.4(b)]{Eis95} and \cite{Gab59}):
\[
\displaystyle
E_R(K) \cong \D:=K_{DP}[X_1,\ldots,X_m]=\bigoplus_{i \geq 0}\Hom_K(R_i,K),
\]
called the \emph{divided power ring}.
The ring $\D$ has a structure of $R$-module via the following contraction action:
\begin{equation}\label{Eqn:DividedPower}
f \circ F=\begin{cases}
\sum_{\nn,\nn'}\alpha_{\nn}\beta_{\nn'} X_1^{n'_1-n_1}\cdots X_m^{n'_m-n_m}& \mbox{ if } \nn' \geq \nn\\  
0 & \mbox{ otherwise }
          \end{cases}
\end{equation}
where $f=\sum_{\nn \in \N^m} \alpha_{\nn} x_1^{n_1}\cdots x_m^{n_m}$ and $F=\sum_{\nn' \in \N^m} \beta_{\nn'} X_1^{n_1'}\cdots X_m^{n_m'}$. Moreover, if $z_1,\ldots,z_m$ is a regular system of parameter of $R$ and $Z_1,\ldots,Z_m$ are the corresponding dual elements, that is, $z_i \circ Z_j=\delta_{ij},$ then $\D=K_{DP}[Z_1,\ldots,Z_m]$ (see \cite{Gab59}). 
$\D$ can be equipped with an internal multiplication which is commutative. But throughout this paper whenever we write $X^\nn F$ for $F \in \D$ we mean formal multiplication and not the multiplication in $\D.$ 

If $I \subseteq R$ is an ideal of $R,$ then $(R/I)^{\vee}:=\Hom_R(R/I,E_R(K))$ can be identified with the $R$-submodule $I^\perp$ of $\D$ which is defined as 
\[
 I^\perp:=\{F \in \D:f \circ F=0 \mbox{ for all }f\in I\}.
\]
This submodule is called the \emph{Macaulay's inverse system of $I$}. On the other hand, given an $R$-submodule $W$ of $\D,$ $W^\vee:=\Hom_R(W,E_R(K))$ is the ring $R/\ann_R(W)$ where 
\[
 \ann_R(W):=\{f \in R:f \circ F=0 \mbox{ for all }F \in W\}
\]
is an ideal of $R.$ If $I$ is a homogeneous ideal (resp. $W$ is generated by homogeneous polynomials) then $I^\perp$ is generated by homogeneous polynomials of $\D$ (resp. $\ann_R(W)$ is a homogeneous ideal of $R$). 

By Matlis duality, $R/I$ is Artinian (resp. an $R$-submodule $W$ of $\D$ is finitely generated) if and only if $ (R/I)^\vee \cong I^\perp $ is finitely generated (resp. $W^\vee \cong R/\ann_R(W) $ is Artinian) (see \cite[Section 3.2]{BH98}). Moreover, $\type(R/I)=\mu((R/I)^\vee)$ where $\mu(-)$ denotes the number of minimal generators as $R$-module.  
Emsalem in \cite[Proposition 2]{Ems78} and Iarrobino in \cite{Iar94}, based on the work of Macaulay in \cite{Mac1916}, gave a more precise description of the inverse system of Artinian level $K$-algebras:

\begin{proposition}  
\label{Prop:Emsalem-Iarrobino}
	There is a one-to-one correspondence between ideals $I$ such that $R/I$ is an Artinian level local (resp. graded) $K$-algebra of socle degree $s$ and type $\tau$ and $R$-submodules of $\D$ generated by $\tau$ polynomials (resp. homogeneous polynomials)  of degree $s$ having linearly independent forms of degree $s.$ The correspondence is defined as follows:
	\begin{eqnarray*}
		\left\{  \begin{array}{cc} I \subseteq R \mbox{ such that } R/I \mbox{ is an Artinian } \\
			\mbox{level (resp. graded) local $K$-algebra of }\\
			\mbox{ socle degree $s$ and type } \tau 
		\end{array} \right\}  \ &\stackrel{1 - 1}{\longleftrightarrow}& \ 
		\left\{ \begin{array}{cc} W \subseteq \D \mbox{ submodule generated by } \tau \mbox{ polynomials }\\
			\mbox{(resp. homogeneous polynomials) of degree} \\
			s  \mbox{ with l.i. forms  of degree } s 
		\end{array} \right\} \\
		R/I \ \ \ \ \ &\longrightarrow& \ \ \ \ \ I^\bot  \ \ \ \ \ \ \ \ \qquad \ \ \ \  \ \ \ \qquad \ \ \ \  \ \ \ \\
		R/\ann_R(W) \ \ \ \ \  &\longleftarrow& \ \ \ \ \  W \ \ \ \ \ \ \ \ \ \ \ \ \ \ \ \  \qquad \qquad \ \ \ \  \ \ \
	\end{eqnarray*}
\end{proposition}
In particular, an Artinian Gorenstein $K$-algebra of socle degree $s$ corresponds to a cyclic $R$-submodule 
generated by a non-zero polynomial $F$ of degree $s$ in $\D.$ 

\vskip 2mm
Consider the following exact pairing of $K$-vector spaces induced by the contraction action:
\[
\xymatrix@R=.3pc{
	\langle\ , \ \rangle: R \times \D \ar[r] & K \\
	\qquad (f,F) \ar@{|->}[r]& (f\circ F)(0).
}
\]
Since $(\mathcal M I \circ F)(0)=0$ if and only if $I \circ F=0,$ it follows that
\[
 I^\perp=\{F \in \D:\langle f,F \rangle \mbox{ for all }f \in I\}.
\]

Using this bilinear pairing we can deduce the following property of $\ann_R(W)$ (see \cite[Proposition 12.9]{Cil94}).
 
\begin{proposition}
\label{Prop:BasicPropertiesOfAnn}
 Let $W_1, ~W_2$ be finitely generated $R$-submodules of $\D.$ Then 
 \[
  \ann_R(W_1 \cap W_2)=\ann_R(W_1) + \ann_R(W_2).
 \]

\end{proposition}

\vskip 2mm 
If $\{F_1,\ldots,F_t\} \subseteq \D$ is a set of polynomials, we will denote by $\langle F_1,\ldots,F_t \rangle$ the $R$-submodule of $\D$ generated by $F_1,\ldots,F_t$, i.e., the $K$-vector space spanned by $F_1,\ldots,F_t$ and by the corresponding derivatives of all orders. 
If $W \subseteq \D$ is generated by polynomials $F_1,\ldots,F_t,$ with $F_i \in \D, $ then we write $\ann_R(W)= \ann_R\langle F_1,\ldots,F_t\rangle .$

\section{Inverse System of Level $K$-algebras}
\label{Section:MainResult}

In \cite{ER17} authors established a one-to-one correspondence between  Gorenstein $K$-algebras of dimension $d>0$ and $G_d$-admissible submodules of $\D.$ The main goal of this section is to extend their result and give the structure of the inverse system of level $K$-algebras of positive dimension.

First we construct the inverse system of a $d$-dimensional level $K$-algebra.
Then, using the properties of this inverse system, we define $L^\tau_d$--admissible submodule of $\D$. Finally, we show that the Matlis dual of an $L^\tau_d$-admissible submodule is indeed a level ring.

\vskip 2mm
We fix some notations. Let $R=K[x_1,\ldots,x_m]$ and $\uz$ a regular linear sequence for $A=R/I$, or $R=K[\![x_1,\ldots,x_m]\!]$ and $\uz:=z_1,\ldots,z_d $ a sequence of general elements in $ \mathcal{M}$. 
If $R=K[\![x_1,\ldots,x_m]\!]$, then by Fact \ref{Fact:GeneralElements} $(\uz)$ is a minimal reduction of $\mathcal{M}.$ Therefore from \cite[Corollary 8.3.6]{HS06} it follows that $\uz$ is part of a regular system of parameter for $R.$ Hence in both cases $\uz$ can be extended to a minimal system of generator of $\mathcal{M},$ say $z_1,\ldots,z_d,\ldots,z_m.$ We write $Z_1,\ldots,Z_m$ for the dual elements in $\D$ corresponding to $z_1,\ldots,z_m,$ that is, $z_i \circ Z_j=\delta_{ij}.$ Hence $\D=K[Z_1,\ldots,Z_m]$ (see \cite{Gab59}).  

We recall the following proposition from \cite{ER17} that is needed to construct the inverse system of level $K$-algebras. Notice that the proof of this given in \cite{ER17} does not depend on the ring being Gorenstein, and hence we state it without proof.  

\begin{proposition}\cite[Proposition 3.2]{ER17} 
\label{Prop:DualExactSeq}
Let $R=K[\![x_1,\ldots,x_m]\!]$ (resp. $R=K[x_1,\ldots,x_m]$) and $A=R/I$ be Cohen-Macaulay of dimension $d>0.$ Let $\uz:=z_1,\ldots,z_d \in \mathcal{M}$ be a sequence of general elements (resp. regular linear sequence for $A$).
For any $\nn \in \N^d_+, $ let $T_{\nn}=(I+(\uz^\nn))^{\perp}.$
\begin{enumerate}
 \item \label{Prop:DualExactSeq:d=1}
 Assume $d=1.$ Then for all $n \geq 2,$ there is an exact sequence of finitely generated $R$-submodules of $\D$
\[
  0 \longrightarrow T_{1} \longrightarrow T_{n} \xrightarrow{z_1\circ} T_{{n-1}} \longrightarrow 0.
 \]
\item \label{Prop:DualExactSeq:d>1}
Assume $d \geq 2.$ For all $\nn \in \N^d_+$ such that $\nn \geq {\bf 2}_d$ there is an exact sequence of finitely generated $R$-submodules of $\D$ 
\[
\displaystyle
  0 \to  T_{{\bf 1}_d} \to  T_{\nn} \to  \bigoplus_{k=1}^d T_{{\nn-\ee_k}} \to  \bigoplus_{1 \leq i < j \leq d } T_{{\nn-\ee_i-\ee_j}}.
\]
\end{enumerate}
 \end{proposition}

We recall the following basic lemma, which will be used later. 
 
\begin{lemma}
\label{Lemma:Generates}
Let $f:M\to N$ be an epimorphism between two non-zero $R$-modules  $M$ and $N$ 
that are minimally generated by $\nu$ elements. 
Let $m_1,\ldots,m_\nu$ be such that $f(m_1),\ldots,f(m_\nu)$ generate $N$. 
Then $m_1,\ldots,m_\nu$ generate $M$.
\end{lemma}
\begin{proof} \pushQED{\qed}
As $f: M \to N$ is surjective, $\bar{f}:M/\M M \to N/\M N$ is also surjective. Since $N/\M  N$ is generated by $f(m_1)+\M N, \ldots,f(m_\nu)+\M N$ as a $K$-vector space and $\dim_K N/\M N=\nu,$  $f(m_1)+\M N, \ldots,f(m_\nu)+\M N$ are $K$-linearly independent. This implies that $m_1+\M M,\ldots,m_\nu + \M M$ are also $K$-linearly independent and hence $m_1+\M M,\ldots,m_\nu + \M M$ generate $M/\M M. $ Therefore by Nakayama's lemma $m_1,\ldots,m_\nu$ generate $M.$ \qedhere
\end{proof} 
 
In the following proposition we construct the inverse system of level $K$-algebras satisfying certain conditions.

\begin{proposition}
\label{Prop:ConstructionOfAdmissible}
Let $R=K[\![x_1,\ldots,x_m]\!]$ with $\mathrm{char}(K)=0$ (resp. $R=K[x_1,\ldots,x_m]$), $A=R/I$ and $\uz:=z_1,\ldots,z_d \in \mathcal{M}$ a sequence of general elements (resp. regular linear sequence for $A=R/I$). 
Let $s$ be the index of nilpotency (resp. Castelnuovo-Mumford regularity) of $A.$ 
Suppose that $A$ is a level $K$-algebra of dimension $d>0$ and type $\tau.$ Then there exists a system of (resp. homogeneous) generators $\{H_\nn^j:j=1,\dots,\tau,\nn \in \N_+^d\}$ of $I^\perp$ such that 
\begin{asparaenum}
 \item for all $\nn \in \N^d_+$ 
 \[
 W_{\nn}:=(I+(\uz^\nn))^\perp=\langle H_\nn^j : j=1,\dots,\tau\rangle
 \]  
 where $\deg H_\nn^1=\deg H_\nn^2=\cdots=\deg H_\nn^\tau=s+|\nn|-d$ and the forms of degree $s+|\nn|-d$ of $H_\nn^1,\ldots,H_\nn^\tau$ are linearly independent; 
 \item for all $\nn \in {\mathbb{N}_+^d}$ and $j=1,\ldots,\tau$ and $i=1,\ldots,d$
  \begin{equation*}
   z_i \circ H_\nn^j =\begin{cases}
                     H_{\nn-\ee_i}^j & \mbox{ if } \nn-\ee_i >{\mathbf{0}_d}\\
                     0 & \mbox{ otherwise}.
                    \end{cases}
  \end{equation*}
\end{asparaenum}
\end{proposition}
\begin{proof}
We prove proposition for the local case. The proof is similar for the graded case. 

For all $\nn \in \N^d_+,$ let $I_{\nn}:=I+(\uz^\nn),$ $R_\nn:=R/I_{\nn}=A/(\uz^\nn)$ and $W_{\nn}:=I_{\nn}^\perp.$ Since $\uz$ is a sequence of general elements in $\mathcal{M},$ $\uz +I:=z_1+I,\ldots,z_d+I$ is a sequence of general elements in $\m.$ Therefore by Fact \ref{Fact:GeneralElements}, $(\uz)A$ is general minimal reduction of $\m.$ Since $A$ is level, $R_{\mathbf{1}_d}$ is an Artinian local level $K$-algebra of type $\tau.$ 
Note that $\socdeg (R_{\mathbf{1}_d})=s.$ Hence by 
Proposition \ref{Prop:Emsalem-Iarrobino} there exist polynomials $H_{{\bf 1}_d}^1,  H_{{\bf 1}_d}^2,\ldots,H_{{\bf 1}_d}^\tau$ of degree $s$ such that the forms of degree $s$ of $H_{{\bf 1}_d}^1,\ldots, H_{{\bf 1}_d}^\tau$ are linearly independent  and 
\[W_{{\bf 1}_d}=(I+(\uz))^\perp=\langle H_{{\bf 1}_d}^j: j=1,\ldots, \tau \rangle.\] 
For $\nn=(n_1,\ldots,n_d)\in \N^d_+, $ let $|\nn|_+$ denote the cardinality of the set $\{n_i:n_i \geq 2\}.$  
We put the lexicographic order on the set $ \{1,\ldots,d\} \times \N_+,$ that is $(i_1,j_1) < (i_2,j_2)$ if $i_1 < i_2 $ OR if $i_1=i_2$ and $ j_1 < j_2.$ We construct $H_\nn^j,$ $\nn\in \N^d_+,~j=1,\ldots,\tau,$ as required using induction on the pair $(|\nn|_+,|\nn|-d+|\nn|_+) \in \{1,\ldots,d\} \times \N_+$.
Assume that $|\nn|_+=1.$ After a permutation, we may assume that $\nn=(n,1,\ldots,1)$ with $n \geq 2.$ 
Since $|\nn-\ee_1|_+\leq 1,$ we have
\[
 (|\nn-\ee_1|_+,|\nn-\ee_1|-d+|\nn-\ee_1|_+) < (|\nn|_+,|\nn|-d+|\nn|_+).
 \]
 Hence by induction for all $\nn' \leq \nn-\ee_1$ and $j=1,\dots,\tau$ there exist $H_{\nn'}^j \in \D$ such that $\{H_{\nn'}^j:{\nn' \in \mathbb{N}_+^d,~ j=1,\ldots,\tau,\nn' \leq \nn-\ee_1}\}$ satisfy the required conditions.
Let $J=I+(z_2,\ldots,z_d).$ 
Now, by Proposition \ref{Prop:DualExactSeq}\eqref{Prop:DualExactSeq:d=1}, we get an exact sequence
\[
 0 \longrightarrow W_{{\bf 1}_d}= (J+(z_1))^\perp\longrightarrow W_{\nn}=(J+(z_1^n))^\perp \overset{z_1 \circ}{\longrightarrow} W_{\nn-\ee_1}= (J+(z_1^{n-1}))^\perp\longrightarrow 0.
\]
Therefore for each $j=1,\ldots,\tau,$ there exist polynomials $H_\nn^j \in W_{\nn}$ such that 
\begin{equation}
\label{Eqn:Induction1}
z_1 \circ H_\nn^j = H_{\nn-\ee_1}^{j}. 
\end{equation}
As $z_i \in I_{\nn}$ for $i \neq 1,$ $z_i \circ H_{\nn}^j=0$ if $i \neq 1.$ 

\noindent  Since $\socdeg (R_\nn)=s+|\nn|-d$ by Proposition \ref{Prop:SocDegOfR_nn}, $\deg H_{\nn}^j \leq s+|\nn|-d$ for all $j=1,\ldots,\tau.$ Moreover, as $\deg H_{\nn-\ee_1}^j=s+|\nn|-d-1,$ by \eqref{Eqn:Induction1} $\deg H_{\nn}^j \geq s+|\nn|-d.$ Hence for all $j=1,\ldots,\tau$ 
\[
\deg H_{\nn}^j = s+|\nn|-d.
\]  
Also, since the forms of degree $s+|\nn|-d-1$ of $H_{\nn-\ee_1}^1,\ldots,H_{\nn-\ee_1}^\tau$ are linearly independent, the forms of degree $s+|\nn|-d$ of $H_\nn^1,\ldots,H_\nn^\tau$ are also linearly independent. 

\noindent
By Proposition \ref{Prop:Emsalem-Iarrobino} $W_\nn$ and $W_{\nn-\ee_1}$ are minimally generated by $\tau$ elements. Hence by Lemma \ref{Lemma:Generates}  $W_\nn=\langle H_\nn^j:j=1,\ldots,\tau \rangle.$ 
\vskip 2mm
Let $l:=|\nn|_+ \geq 2.$ After a permutation, we may assume that $\nn=(n_1,\ldots,n_l,1,\ldots,1)$ with $n_i \geq 2$ for $i=1,\ldots,l.$ We set $\uz'=z_1,\ldots,z_l$ and $\nn'=(n_1,\ldots,n_l)\in \N^l_+$ and $J=I+(z_{l+1},\ldots,z_d).$ By Proposition 
\ref{Prop:DualExactSeq}\eqref{Prop:DualExactSeq:d>1} we get an exact sequence
\begin{equation}\label{Eqn:exactSequence}
0 \longrightarrow T_{{\bf 1}_l} \longrightarrow T_{\nn'} \longrightarrow  \bigoplus_{k=1}^l T_{\nn'-\ee_k} \overset{\phi^*_{\nn'}}{\longrightarrow} \bigoplus_{1 \leq i < j \leq l } T_{{\nn'-\ee_i-\ee_j}}.
\end{equation}
where $T_{{\bf 1}_l}=W_{{\bf 1}_d}, T_{\nn'}=(J+(\uz'^{\nn'}))^\perp=W_\nn,T_{\nn'-\ee_k}=(J+(\uz'^{\nn'-\ee_k}))^\perp=W_{\nn-\ee_k}$ and 
$T_{\nn'-\ee_i-\ee_j}=(J+(\uz'^{\nn'-\ee_i-\ee_j}))^\perp=W_{\nn-\ee_i-\ee_j}$ and $\phi^*_{\nn'}$ is 
defined as
\[
 \phi^*_{\nn'}(v_1,\ldots,v_l)=(z_k \circ v_i-z_i \circ v_k: 1\leq i<k\leq l).
\]
Since for all $k=1,\ldots,l$
\[
 (|\nn-\ee_k|_+,|\nn-\ee_k|-d+|\nn-\ee_k|_+) < (|\nn|_+,|\nn|-d+|\nn|_+),
 \]
by induction there exist $H_{\nn-\ee_k}^j \in \D,$ $j =1,\ldots, \tau,$ such that 
\begin{itemize}
 \item $W_{\nn-\ee_k}=\langle H_{\nn-\ee_k}^j:j=1,\ldots,\tau\rangle,	$ $\deg H_{\nn-\ee_k}^j=s+|\nn|-d-1$ for all $j=1,\ldots,\tau$ and the forms of degree $s+|\nn|-d-1$ of $H_{\nn-\ee_k}^1,\ldots,H_{\nn-\ee_k}^\tau$ are linearly independent;
  \item $z_i \circ H_{\nn-\ee_k}^j=H_{\nn-\ee_k-\ee_i}^j $ for all $i=1,\ldots,l$ and $i \neq k.$
  \end{itemize}
Therefore for each $j=1,\ldots,\tau,$ 
\begin{eqnarray*}
 \phi^*_{\nn'}(H_{\nn-\ee_1}^j,\ldots,H_{\nn-\ee_l}^j)&=&(z_k \circ H_{\nn-\ee_i}^j - z_i \circ H_{\nn-\ee_k}^j: 1\leq i<k\leq l)\\
 &=&(H_{\nn-\ee_i-\ee_k}^j - H_{\nn-\ee_k-\ee_i}^j:1\leq i<k\leq l)=0.
\end{eqnarray*}
Hence
\[
 (H_{\nn-\ee_1}^j,\ldots,H_{\nn-\ee_l}^j) \in \ker (\phi^*_{\nn'}).
 \]
Therefore by exactness of the sequence \eqref{Eqn:exactSequence}, we conclude that there exist $H_\nn^j \in W_\nn$ such that for all $k=1,\ldots,l,$
 \begin{equation}
  \label{Eqn:Induction2}
z_k \circ H_\nn^j=H_{\nn-\ee_k}^j.  
 \end{equation}
Moreover, as $z_k \in I_{\nn}$ for $k>l,$ $z_k \circ H_{\nn}^j=0$ if $k >l.$

\noindent
Since $\socdeg (R_\nn)=s+|\nn|-d$ by Proposition \ref{Prop:SocDegOfR_nn}, $\deg H_{\nn}^j\leq s+|\nn|-d.$ Also, as $\deg H_{\nn-\ee_k}^j=s+|\nn|-d-1$, by \eqref{Eqn:Induction2} $\deg H_\nn^j\geq s+|\nn|-d.$ Hence $\deg H_{\nn}^j=s+|\nn|-d$ for all $j=1,\ldots,\tau.$ 
As the forms of degree $s+|\nn|-d-1$ of  $H_{\nn-\ee_1}^1,\ldots,H_{\nn-\ee_1}^{\tau}$ are linearly independent, the forms of degree $s+|\nn|-d$ of $H_\nn^1,\ldots,H_\nn^\tau$ are also linearly independent. Applying Lemma \ref{Lemma:Generates} to the map 
 \[
  W_{\nn} \overset{z_1 \circ}{\longrightarrow} W_{\nn-\ee_1}
 \]
we conclude that $W_\nn=\langle H_\nn^1,\ldots,H_\nn^\tau\rangle.$ Thus we have constructed $H_\nn^j,$ for all $j=1,\ldots,\tau$ and $\nn \in \N^d_+$ satisfying the required conditions. 

Moreover, $I^\perp=\langle W_{\nn} :\nn \in \N^d_+ \rangle=:W.$ Indeed, if $F \in I^\perp,$ then $(I_{{\bf t}_d}) \circ F=0$ where $t>\deg F.$ Hence $F \in W_{{\bf t}_d} \subseteq W.$ The other containment is trivial.
\qedhere
\end{proof}

Motivated by Proposition \ref{Prop:ConstructionOfAdmissible} we define $L_d^\tau$-admissible $R$-submodules of $\D$ as follows:

\begin{definition}
\label{Def:LdAdmissible}
 Let $d$ and $\tau$ be positive integers. An $R$-submodule $W$ of $\D$ is called {\it local} (resp. {\it graded}) {\it $L_d^\tau$-admissible} if it admits a system of (resp. homogeneous) generators $\{H_\nn^j:{\nn \in \mathbb{N}_+^d, j =1,\ldots, \tau}\}$ satisfying the following conditions:
 \begin{enumerate}[(1)]
 \item \label{Def:Cond1} for all $\nn \in \N_+^d,$ $s_{\nn}:=\deg H_\nn^1=\deg H_\nn^2=\cdots=\deg H_\nn^\tau$ and the forms of degree $s_{\nn}$ of $H_\nn^1,\ldots,H_\nn^\tau$ are linearly independent;
  \item \label{Def:Cond2} there exists a sequence of general elements (resp. regular linear sequence) $z_1,\ldots,z_d $ in $\mathcal{M}$ such that 
  \begin{equation}
  \label{Eqn:Cond2}
   z_i \circ H_\nn^j =\begin{cases}
                     H_{\nn-\ee_i}^j & \mbox{ if } \nn-\ee_i >\mathbf{0}_d\\
                     0 & \mbox{ otherwise}
                    \end{cases}
  \end{equation}
  for all $\nn \in {\mathbb{N}_+^d},$ $j=1,\ldots,\tau$ and $i=1,\ldots,d;$
\item \label{Def:Cond3} for all $i=1,\ldots,d$ and ${\nn \in \mathbb{N}_+^d} $ such that $\nn-\ee_i > \mathbf{0}_d$
\begin{equation}
\label{Eqn:Cond3}
W_{\nn} \cap V_\nn^i \subseteq W_{\nn-\ee_i} 
\end{equation}
where $W_\nn=\langle H_\nn^j:j=1,\ldots,\tau\rangle$ and $V_\nn^i:= \langle Z_1^{k_1}\cdots Z_m^{k_m} : \kk=(k_1,\ldots,k_m) \in \N^m \mbox{ with } k_i\leq n_i-2 \mbox{ and }|\kk| \leq s_{\nn} \rangle.$ 
 \end{enumerate}
If this is the case, we also say that $W$ is $L_d^\tau$-admissible with respect to the sequence $\uz:=z_1,\ldots,z_d.$
\end{definition} 

\begin{remark}
\label{Remark:Definition}
\begin{asparaenum}
\item We remark that if an $R$-submodule $W$ of $\D$ satisfies \eqref{Eqn:Cond2}, then 
\[
 W_{\nn-\ee_i} \subseteq W_{\nn} \cap V_\nn^i 
\]
for all $i=1,\ldots,d$ and $\nn-\ee_i > \mathbf{0}_d.$ Indeed, by \eqref{Eqn:Cond2} $W_{\nn-\ee_i} \subseteq W_{\nn}$ and 
\[
z_i^{n_i-1} \circ H_{\nn-\ee_i}^j=0
\] for all $j=1,\ldots,\tau$ which implies that 
$W_{\nn-\ee_i} \subseteq V_{\nn}^i.$ Thus $W_{\nn-\ee_i} \subseteq W_{\nn} \cap V_\nn^i.$
Hence if $W$ is $L_d^\tau$-admissible, then equality holds in \eqref{Eqn:Cond3}.

\item 
\label{Remark:vanishing}
Let $\{H_\nn^j:{1 \le j \le \tau, \nn \in \N_+^d}\}$ be a local or graded $L_d^\tau$-admissible $R$-submodule of $\D.$ Put $W_\nn:=\langle H_\nn^j:j=1,\ldots,\tau\rangle.$ Then
$W_{{\bf 1}_d}=0$ if and only if $W_\nn=0$ for all $\nn\in \N^d_+.$ \\
{\bf Proof:} Suppose $W_{{\bf 1}_d}=0.$ We use induction on $t=|\nn|$ to show that $W_\nn=0$ for all $\nn\in \N^d_+.$ If $t=d,$ then $\nn={\bf 1}_d$ and by assumption $W_{{\bf 1}_d}=0.$ Assume that $t >d $ and $W_{\nn}=0$ for all $\nn$ with $|\nn| \leq t.$ It suffices to show that $W_{\nn+\ee_i}=0$ for all $i=1,\ldots,d.$ From \eqref{Eqn:Cond3} it follows that $W_{\nn+\ee_i} \cap V_{\nn+\ee_i}^i \subseteq W_\nn=0$. If $W_{\nn+\ee_i} \neq 0,$ then $0\neq 1 \in W_{\nn+\ee_i} \cap V_{\nn+\ee_i}^i$ which is a contradiction. Hence $W_{\nn+\ee_i}=0$ for all $i=1,\ldots,d.$ 
The converse is trivial.  
\end{asparaenum}
\end{remark}

In the following proposition we prove that if $W$ is an $L_d^\tau$-admissible $R$-submodule of $\D,$ then $R/\ann_R(W)$ is a level $K$-algebra. We do not need assumptions on the characteristic for this.

\begin{proposition}
\label{Prop:AdmissibleGivesLevel}
Let $W=\langle H_\nn^j:j=1,\dots,\tau, \nn \in \N_+^d \rangle$ be a non-zero local or graded $L_d^\tau$-admissible $R$-submodule of $\D.$ Then $R/\ann_R(W)$ is a level $K$-algebra of dimension $d>0$ and type $\tau.$
\end{proposition}
\begin{proof}
We prove the result for the local case. The proof is similar for the graded case. We give details for the graded case, if necessary.

Let $W$ be $L_d^\tau$-admissible with respect to a sequence of general elements $\uz:=z_1,\ldots,z_d.$ 
We set $W_\nn:=\langle H_\nn^j:j=1,\ldots,\tau\rangle ,$ $I_\nn:=\ann_R(W_\nn) $ and $I:=\ann_R(W).$ Hence   $I=\cap_{\nn \in \N^d_+} I_\nn.$ \\
{\bf Claim 1:} $I_\nn=I+({\uz^\nn})$ for all $\nn \in \N^d_+.$\\
{\bf Proof of Claim 1:} 
First we prove that for all $\nn\in \N^d_+,$
\begin{equation}
\label{Eqn:containment}
I_\nn \subseteq I_{\nn+{\bf 1}_d}+(\uz^\nn).
\end{equation}
Note that for all $i=1,\ldots,d$ and $\nn \in \N_d^+,$ 
\[
\ann_R(V_{\nn}^i)=(z_i^{n_i-1})+\mathcal{M}^{s_{\nn}+1}.
\]
Indeed, suppose $f \in \ann_R(V_{\nn}^i).$ Let 
$  f=\sum_{\kk \in \N^m} \alpha_{\kk} x^\kk. $ Since $x^{\kk} \in \mathcal{M}^{s_{\nn}+1}$ if $|\kk| > s_{\nn},$ we may assume that $f$ is a polynomial of degree at most $s_{\nn}.$ Write
\[
 f=z_i^{n_i-1} f_1+f_2
\]
for some $f_1,f_2\in R$ such that $\deg _{~z_i} f_2 \leq n_i-2 .$ As $\deg f \leq s_{\nn},$ $\deg f_2 \leq s_{\nn}.$ 
Assume that $f_2 \neq 0.$ Then there exists $G \in V_{\nn}^i$ such that $f_2 \circ G \neq 0.$ But as $f,z_i^{n_i-1} \in \ann_R(V_{\nn}^i),$ $f_2 \in \ann_R(V_{\nn}^i)$ and hence $f_2 \circ G=0,$ which leads to a contradiction. Therefore $f_2=0$ and hence $f \in (z_i^{n_i-1}).$ Thus 
$\ann_R(V_{\nn}^i)\subseteq(z_i^{n_i-1})+\mathcal{M}^{s_{\nn}+1}.$ The other containment is clear.

Now, since $W_{\nn+\ee_i} \cap V_{\nn+\ee_i}^i \subseteq W_{\nn},$ by using Proposition \ref{Prop:BasicPropertiesOfAnn} we get 
\begin{eqnarray*}
I_{\nn}= \ann_R(W_{\nn}) \subseteq  \ann_R(W_{\nn+\ee_i} \cap V_{\nn+\ee_i}^i)=\ann_R(W_{\nn+\ee_i})+\ann_R(V_{\nn+\ee_i}^i)=I_{\nn+\ee_i}+(z_i^{n_{i}}).
\end{eqnarray*}
The last equality follows because $\mathcal{M}^{s_{\nn+\ee_i}+1} \subseteq I_{\nn+\ee_i}.$ 
Therefore for all $\nn\in \N^d_+$ 
\begin{eqnarray*}
 \nonumber I_{\nn} \subseteq I_{\nn+\ee_1}+(z_1^{n_1})\subseteq I_{\nn+\ee_1+\ee_2}+(z_1^{n_1},z_2^{n_2}) \subseteq \cdots\\
 \cdots \subseteq I_{\nn+\ee_1\cdots+\ee_d}+(z_1^{n_1},\ldots,z_d^{n_d})=I_{\nn+{\bf 1}_d}+ (\uz^{\nn}).
 \end{eqnarray*}
Fix $\nn \in \N^d_+$ and consider $f \in I_\nn.$ By \eqref{Eqn:containment} there are $f_{\nn+{\bf 1}_d} \in I_{\nn+{\bf 1}_d}$ and $g_0 \in (\uz^\nn)$ such that 
\[
 f=f_{\nn+{\bf 1}_d}+g_0.
\]
Since $f_{\nn+{\bf 1}_d} \in I_{\nn+{\bf 1}_d},$ again by \eqref{Eqn:containment} there are $f_{\nn+{\bf 2}_d} \in I_{\nn+{\bf 2}_d}$ and $g_1 \in (\uz^{\nn+{\bf 1}_d})$ such that
\[
 f_{\nn+{\bf 1}_d}=f_{\nn+{\bf 2}_d}+g_1.
\]
Thus $f=f_{\nn+{\bf 2}_d}+g_0+g_1.$ By recurrence there are sequences $\{f_{\nn+{\bf t}_d}\}_{t \geq 0}$ and $\{g_t\}_{t \geq 0}$ such that $f_{\nn+{\bf t}_d} \in I_{\nn+{\bf t}_d},$ $g_t \in (\uz^{\nn+{\bf t}_d})$ and $f_{\nn+({\bf t-1})_d}=f_{\nn+{\bf t}_d}+g_{t-1}.$ For all $t \geq 0, $ it holds
\begin{equation}
\label{Eqn:FormulaForf}
 f=f_{\nn+{\bf t}_d}+\sum_{i=0}^{t-1}g_i.
\end{equation}
Let ${g'}=\sum_{i \geq 0}g_i \in K[\![x_1,\ldots,x_m]\!].$ Let ${f'}=\lim_{t \to \infty} f_{\nn+{\bf t}_d} \in K[\![x_1,\ldots,x_m]\!].$ Taking limit as $t \to \infty$ in \eqref{Eqn:FormulaForf}, we get  
\begin{eqnarray*}
 f={f'}+{g'}.
\end{eqnarray*}
Since $g_t \in (\uz^{\nn+{\bf t}_d})$ for all $t\geq 0,$ ${g'} \in (\uz^\nn).$ 
Since for every ${\bf k} \in \N^d_+,$ $W_{\kk}$ is finitely generated $R$-submodule of $\D$, $R/I_{\kk}$ is Artinian. Hence there exists a positive integer $N(\kk)$ such that $\M^{N(\kk)} \subseteq I_{{\kk}}.$ 
Note that as $W_{\kk} \subseteq W_{\kk+{\bf t}_d}$ by \eqref{Eqn:Cond2}, $I_{{\bf k}+{\bf t}_d} \subseteq I_{{\bf k}}$ for all $t \geq 0.$  
Since $f_{{\bf k}+{\bf t}_d} - {f'} \in \M^{N(\kk)} \subseteq I_{{\bf k}}$ for all $t \gg 0$ and $f_{{\bf k}+{\bf t}_d}\in I_{{\bf k}+{\bf t}_d} \subseteq I_{{\bf k}}$ for all $t \geq 0,$ we get that ${f'} \in I_{{\bf k}}$ for all ${\bf k} \in \N^d_+.$ Thus ${f'} \in I$ and hence $f \in I+(\uz^\nn).$ This gives that $I_\nn \subseteq I+(\uz^\nn).$

If $R=K[x_1,\ldots,x_m], $ then ${f'} \in I \subseteq R.$ Since $f \in R$ we get that 
${g'}=\sum_{i \geq 0}g_i \in R.$

On the other hand, by \eqref{Eqn:Cond2} $z_i^{n_i} \circ H_\nn^j =0$ for all $j=1,\ldots,\tau$ and $i=1,\ldots,d.$ Hence $(\uz^\nn) \subseteq I_\nn.$ Clearly, $I \subseteq I_\nn.$ Therefore $I+(\uz^\nn) \subseteq I_\nn.$ This proves Claim 1.
\vskip 2mm
\noindent {\bf Claim 2:} $\uz$ is a regular sequence modulo $I$ and $\dim(R/I)=d.$\\
{\bf Proof of Claim 2:} By Claim 1 $\ann_R(W_{{\bf 1}_d})=I_{{\bf 1}_d}= I+(\uz).$ Since $W \neq 0,$ by Remark \ref{Remark:Definition}\eqref{Remark:vanishing} $W_{{\bf 1}_d} \neq 0.$ Since $W_{{\bf 1}_d}$ is a non-zero finitely generated $R$-submodule of $\D,$ $R/(I+(\uz))$ is Artinian. This implies that $\dim (R/I) \leq d.$ Next we prove that $\uz$ is a regular sequence modulo $I$ and hence $\dim(R/I)=d.$

First we prove that $z_1$ is a nonzero-divisor on $A=R/I.$ By \eqref{Eqn:Cond2} the derivation by $z_1$ defines an epimorphism of $R$-modules
\[
W_{\nn + \ee_1} \overset{z_1\circ}{\longrightarrow} W_{\nn} \longrightarrow 0 
\]
for all $\nn  \in \N^d_+ .$ Since $\ann_R(W_{\nn})= I_{\nn}= I+(\uz^\nn)$ by Claim 1, this sequence induces an exact sequence of $R$-modules
\[
0\longrightarrow \frac{R}{I+(\uz^{\nn})} \overset{\bigcdot z_1}{\longrightarrow} \frac{R}{I+(\uz^{\nn+\ee_1})}
\]
for all $\nn \in \N^d_+$.
Let $f \in R$ be such that $z_1f \in I.$ 
Since $z_1 f \in I+(\uz^{\nn+\ee_1}),$ we deduce that $f \in I+(\uz^{\nn})=I_{\nn}$ for all $\nn \in \N^d_+$ and hence we conclude that $f \in I.$\\
Assume that $z_1,\ldots,z_l$ is a regular sequence of $R/I$, with $l < d$. Given $\nn'=(n_{l+1},\ldots,n_d) \in \N^{d-l}_+,$ 
we take $\nn=(1,\ldots,1,n_{l+1},\ldots,n_d) \in \N^d_+.$ By \eqref{Eqn:Cond2} the derivation by $z_{l+1}$ defines an epimorphism of $R$-modules for all $n_{l+1} \geq 2$ 
\[
W_\nn \overset{z_{l+1}\circ}{\longrightarrow} W_{\nn-\ee_{l+1}} \longrightarrow 0.
\]
Since $\ann_R(W_{\nn})= I_{\nn}= I+(\uz^\nn) $ by Claim 1, this sequence induces an exact sequence of $R$-modules
\[
0 \longrightarrow \frac{R}{I+(z_1,\ldots,z_l)+(z_{l+1}^{n_{l+1}-1},\ldots,z_d^{n_d})} \overset{\bigcdot z_{l+1}}{\longrightarrow} \frac{R}{I+(z_1,\ldots,z_l)+(z_{l+1}^{n_{l+1}},\ldots,z_d^{n_d})}.
\]
Let $f \in R$ be such that $z_{l+1} f \in I+(z_1,\ldots,z_l).$ Since $z_{l+1} f \in I+(z_1,\ldots,z_l)+(z_{l+1}^{n_{l+1}},\ldots,z_d^{n_d}),$ we deduce that $f \in I+(z_1,\ldots,z_l)+(z_{l+1}^{n_{l+1}-1},\ldots,z_d^{n_d})$ for all $n_{l+1} \geq 2.$ 
Hence $\bar f\in R/(I+(z_1,\dots,z_l))$ is such that $\bar f\in  {L_t}:=\overline{(z_{l+1}^{t},\ldots,z_d^{t})}$ for all $t  \geq 1.$ Here $\overline{(\cdot)}$ denotes the image in $R/(I+(z_1,\dots,z_l)).$ Since by Krull intersection theorem $\cap_{t \geq 1} L_t=0,$ $\bar f=0$ and hence $f \in I+(z_1,\ldots,z_l).$
Thus $\uz$ is a regular sequence and we have the claim.
\vskip 2mm 
\noindent {\bf Claim 3:} $R/I$ is a level $K$-algebra of type $\tau$.\\
{\bf Proof of Claim 3:} 
Since $\uz$ is a sequence of general elements in $\mathcal{M},$ $\uz+I:=z_1+I,\ldots,z_d+I$ is a sequence of general elements in $\m.$ Therefore by Fact \ref{Fact:GeneralElements}, $({\uz})A \subseteq \m$ is a general minimal reduction of $\m.$
Since $(I+(\uz))^\perp=(I_{{\bf 1}_d})^\perp=W_{{\bf 1}_d}$ and $W_{{\bf 1}_d}$ is generated by polynomials $H_{{\bf 1}_d}^1,\ldots,H_{{\bf 1}_d}^\tau$ of same degree whose forms of degree $ s_{\bf 1}$ are linearly independent, by Proposition \ref{Prop:Emsalem-Iarrobino} $R/(I+(\uz))$ is an Artinian level $K$-algebra of type $\tau.$ Since $R/I$ is Cohen-Macaulay by Claim 2, $R/I$ is a $d$-dimensional level $K$-algebra of type $\tau$ according to Definition \ref{Def:LocalLevel}.
\end{proof}

Summarizing, we can now give a complete characterization of the inverse system of level $K$-algebras.
 
\begin{theorem}
\label{thm:CharOfLocalLevel}
Let $R=K[\![x_1,\ldots,x_m]\!]$ with $\mathrm{char}(K)=0$ (resp. $R=K[x_1,\ldots,x_m]$) and let $0< d \leq m$. There is a one-to-one correspondence between the following sets:
 \begin{enumerate}[(1)]
  \item $d$-dimensional local (resp. graded) level $K$-algebras of type $\tau;$
  \item non-zero local (resp. graded) $L_d^\tau$-admissible $R$-submodules $W=\langle H_\nn^j: j=1,\dots,\tau,\nn \in \N^d_+\rangle$ of $\D;$
 \end{enumerate}
given by
 \begin{eqnarray*}
\left\{  \begin{array}{cc} I \subseteq R \mbox{ such that } R/I \mbox{ is a local } \\
\mbox{  (resp.  graded) level $K$-algebra of }\\
\mbox{ dimension $d$ and type }\tau  
  \end{array} \right\}  \ &\stackrel{1 - 1}{\longleftrightarrow}& \ 
\left\{ \begin{array}{cc} \mbox{ non-zero local (resp. graded)} L_d^\tau\mbox{-admissible }\\
R\mbox{-submodule } W \mbox{ of } \D   \\
\end{array} \right\} \\
R/I \ \ \ \ \ &\longrightarrow& \ \ \ \ \  I^\perp \\
R/\ann_R(W) \ \ \ \ \  &\longleftarrow& \ \ \ \ \  W 
\end{eqnarray*}
\end{theorem} 
\begin{proof}
Let $A=R/I$ be a local level $K$-algebra of dimension $d$ and type $\tau.$ Let $\uz=z_1,\ldots,z_d $ be a sequence of general elements in $\mathcal{M}.$ 
Let $\{H^j_\nn:\nn \in \N^d_+,~ j =1,\ldots, \tau\}$ be a generating set of $I^\perp$ 
as in Proposition \ref{Prop:ConstructionOfAdmissible}. To prove that $I^\perp$ is $L_d^\tau$-admissible 
it is enough to prove \eqref{Eqn:Cond3}. 
For fixed $i=1,\ldots,d$ and $\nn \in N^d_+,$ let $ F \in W_\nn \cap V_\nn^i$ where $W_{\nn}:=(I+(\uz^{\nn}))^\perp$ and $V_\nn^i:= \langle Z_1^{k_1} \cdots Z_m^{k_m} : \kk =(k_1,\ldots,k_m)\in \N^m \mbox{ with } k_i \leq n_i-2 \mbox{ and }|\kk| \leq s+|\nn|-d \rangle .$ We have 
\begin{eqnarray*}
  (I+(\uz^{\nn-\ee_i})) \circ F& =& I \circ F + (\uz^{\nn-\ee_i}) \circ F\\
  &=& z_i^{n_i-1} \circ F \hspace*{4cm} (\mbox{as } F \in W_\nn)\\
&=& 0 \hspace*{5.1cm} (\mbox{as } F \in V_\nn^i).
  \end{eqnarray*}
 Hence $F \in (I+(\uz^{\nn-\ee_i}))^\perp= W_{\nn-\ee_i}$. This proves \eqref{Eqn:Cond3}. 
Hence   
$I^\perp$ is a local $L_d^\tau$-admissible submodule of $\D$. As $d \geq 1,$ $W_{{\bf 1}_d}\neq 0$ which gives that $I^\perp \neq 0.$ 
\vskip 3mm
Let $W=\langle H_\nn^j:1 \le j \le \tau, \nn \in \N_+^d \rangle$ be a non-zero local $L_d^\tau$-admissible $R$-submodule of $\D.$ By Proposition \ref{Prop:AdmissibleGivesLevel}, $R/\ann_R(W)$ is a level $K$-algebra of dimension $d >0$ and type $\tau.$
\vskip 3mm
Let 
\begin{eqnarray*}\mathcal{C}&:=&\left\{I \subseteq R \mbox{ such that } R/I \mbox{ is a local (or graded) level $K$-algebra of }
\mbox{dimension $d$ and type }\tau  
 \right\}  \mbox{ and }\\ 
\mathcal{C}'&:=&\left\{ \mbox{ local (or graded) } L_d^\tau\mbox{-admissible $R$-submodule }
W= \langle H_\nn^j :j=1,\dots,\tau, \nn \in \N^d_+ \rangle \mbox{ of } \D 
 \right\}. 
\end{eqnarray*}
Consider $\theta: \mathcal{C} \to \mathcal{C}'$ defined as 
\[
\theta(R/I)=I^\perp
\]
and $\theta':\mathcal{C}' \to \mathcal{C}$ defined as
\[
\theta'(W)=R/\ann_R(W).  
\]
We prove that $\theta$ and $\theta'$ are inverses of each other. Let $A=R/I$ be a $d$-dimensional level $K$-algebra. 
Then 
\[
 \theta'\theta(R/I)=R/\ann_R(I^\perp).
\]
Let $\uz:=z_1,\ldots,z_d$ be a sequence of general elements in $\mathcal{M}.$ 
Let $\{H_{\nn}^j:\nn \in \N^d_+, ~j =1,\ldots,\tau\}$ be a system of generators of $I^\perp$ as in Proposition \ref{Prop:ConstructionOfAdmissible}. 
Then $W_{\nn}:=(I+(\uz^\nn))^\perp=\langle H_{\nn}^j:j=1,\ldots,\tau\rangle.$ 
Therefore 
\[
 \ann_R(I^\perp)=\bigcap_{\nn \in \N^d_+} \ann_R(W_\nn)= \bigcap_{\nn \in \N^d_+} \ann_R((I+(\uz^\nn))^\perp)=\bigcap_{\nn \in \N^d_+} (I+(\uz^\nn))=I
\]
where $\ann_R((I+(\uz^\nn))^\perp)=I+(\uz^\nn)$ by Matlis duality. The last equality follows by Krull intersection theorem applied to $R/I.$ Hence $\theta'\theta$ is the identity. 

Let $W= \langle H_\nn^j :j=1,\dots,\tau, \nn \in \N^d_+\rangle$ be a local $L_d^\tau$-admissible $R$-submodule of $\D$ with respect to the sequence $\uz:=z_1,\ldots,z_d.$ 
Then
\[
 \theta \theta'(W)=\ann_R(W)^\perp.
\] 
Let $I:=\ann_R(W),$ $I_{\nn}:=\ann_R \langle H_\nn^j :j=1,\dots,\tau \rangle.$ 
Then by Claim 1 in the proof of Proposition \ref{Prop:AdmissibleGivesLevel} we get $
I_{\nn}=I+(\uz^\nn).$ Hence by Proposition \ref{Prop:ConstructionOfAdmissible} it follows that
\[
I^\perp=\langle (I+(\uz^\nn))^\perp:\nn \in \N^d_+\rangle \langle =I_{\nn}^\perp : \nn \in \N^d_+\rangle.
\] 
Since $I_{\nn}^\perp=\langle H_\nn^j :j=1,\dots,\tau \rangle$ by Matlis duality, 
\[
 \ann_R(W)^\perp = I^\perp=\langle I_{\nn}^\perp:\nn \in \N^d_+\rangle=\langle H_\nn^j :j=1,\dots,\tau, \nn \in \N^d_+\rangle =W.
\]
Hence $\theta \theta'$ is the identity. 
 \qedhere
\end{proof}

We do not know whether we can remove the assumption on the characteristic in Theorem \ref{thm:CharOfLocalLevel}. 

\begin{remark}
 With few modifications in Theorem \ref{thm:CharOfLocalLevel} it is not difficult to obtain the structure of the inverse system of any graded Cohen-Macaulay $K$-algebra. But it would be interesting to understand the inverse system of a Cohen-Macaulay $K$-algebra of positive dimension in the local case.  
\end{remark}

Let $A=R/I$ be a $d$-dimensional local level ring. 
By Theorem \ref{thm:CharOfLocalLevel} the dual module $I^\perp=W=\langle H^j_\nn: j=1,\dots,\tau,\nn\in\N^d_+\rangle$ is an $L_d^\tau$-admissible $R$-submodule of $\D,$ say with respect to $\uz=z_1,\dots,z_d.$ By Theorem \ref{thm:CharOfLocalLevel} $\ann_R(W)=I.$ Hence by using Claim 1 in Proposition \ref{Prop:AdmissibleGivesLevel} we get 
\[
(I+(\uz))^\perp=I_{{\bf 1}_d}^\perp= \langle H^j_{{\bf 1}_1}:j=1,\ldots,\tau\rangle.
\] Therefore by Proposition \ref{Prop:Emsalem-Iarrobino} the index of nilpotency $s(A)=\socdeg (A/(\uz)A)$ coincides with $\deg(H^1_{\bf 1_d})$ ($= \deg(H^j_{\bf 1_d})$ for any $1\leq j \leq \tau$). 
Since $(\uz)$ is a general minimal reduction of $\m,$ by Fact \ref{Fact:GeneralElements} $\uz$ forms a superficial sequence for $\m.$ Hence
the \emph{multiplicity} of $A$ is
\[
e(A)=\dim_K(A/(\uz)A)=\dim_K\langle H^j_{\bf 1_d}: j=1,\dots,\tau\rangle.
\]
If $A$ is a homogeneous level $K$-algebra, then 
for any $\uz:=z_1,\ldots,z_d$ regular linear sequence for $A$
\[
e(A)=\dim_K\langle H^j_{\bf 1_d}:j=1,\dots,\tau\rangle  
\] 
and by Proposition \ref{Prop:SocDegOfR_nn}
\[
 \reg(A)=\deg(H^1_{\bf 1_d})=\cdots=\deg(H^\tau_{\bf 1_d}).
\]

Therefore important informations of the level $K$-algebras are encoded in the first $\tau$ polynomials of the dual module.

\begin{theorem}
\label{thm:ExtraProperties}
Let $R=K[\![x_1,\ldots,x_m]\!]$ with $\mathrm{char}(K)=0$ or $R=K[x_1,\ldots,x_m]$ and let $d \leq m$ be a positive integer. There is a one-to-one correspondence between the following sets: 
\begin{enumerate}[(i)]
\item $d$-dimensional local or graded level $K$-algebras $A=R/I$ of multiplicity $e$ (resp. index of nilpotency in the local case or Castelnuovo-Mumford regularity in the graded case, say $s$);
\item non-zero local or graded $L^\tau_d$-admissible $R$-submodules $W=\langle H^j_\nn: j=1,\dots,\tau,\nn\in\N^d_+\rangle$ of $\D$ such that 
$\dim_K\langle H^j_{\bf 1_d}:j=1,\dots,\tau\rangle=e$ (resp. $\deg H^1_{\bf 1_d}=\cdots=\deg H^\tau_{\bf 1_d}=s$ in both local and graded cases).
\end{enumerate} 
\end{theorem}

We now compare the definition of $G_d$-admissible and $L_d^\tau$-admissible $R$-submodules of $\D.$

\begin{discussion}\pushQED{\qed}
\label{Disc:LdAdmissibleDefinition}
According to \cite[Definition 3.6]{ER17}, a submodule $W$ of $\D$ is said to be $G_d$-admissible if it admits a system of generators $\{H_\nn:{\nn \in \N^d_+}\}$ 
such that 
\begin{asparaenum}
\item for all $\nn \in \N^d_+$ and $i=1,\ldots,d$
\[
 z_i \circ H_{\nn}=\begin{cases}
                    H_{\nn -\ee_i} & \mbox{ if } \nn - \ee_i > \mathbf{0}_d\\
                    0 & \mbox{ otherwise };
                   \end{cases}
\]
 \item for all $i=1,\ldots,d$ and $\nn-\ee_i >\mathbf{0}_d$
\begin{eqnarray}
\label{Eqn:GdAdmissible}
\ann_R\langle H_{\nn-\ee_i}\rangle \circ H_{\nn}=\langle H_{\nn-(n_i-1)\ee_i} \rangle.
\end{eqnarray}
\end{asparaenum}
Inspired by the definition of $G_d$-admissible submodule, one could define an $L_d^\tau$-admissible submodule $W:=\langle H_\nn^j: j=1,\ldots, \tau,\nn \in \N^d_+\rangle$ of $\D$ using the following condition instead of \eqref{Eqn:Cond3}:
\begin{eqnarray}
\label{Eqn:WeakCondition}
\ann_R(W_{\nn-\ee_i}) \circ W_{\nn}=W_{\nn-(n_i-1)\ee_i} \mbox{ for all }i=1,\ldots,d \mbox{ and } \nn-\ee_i >\mathbf{0}_d
\end{eqnarray}
where $W_\nn=\langle H_\nn^j:j=1,\ldots,\tau \rangle.$ 
However, as Example \ref{Example:intersection} shows, if $\tau>1$ an $R$-submodule $W$ of $\D$ satisfying \eqref{Eqn:WeakCondition} need not satisfy \eqref{Eqn:Cond3} and therefore $R/\ann_R(W)$ would not be level with this definition.  

\vskip 3mm
We now prove that our condition \eqref{Eqn:Cond3} is stronger, as it implies \eqref{Eqn:WeakCondition}.
 For $1 \leq j \leq \tau$ consider $f \circ H_\nn^j$ where $f \in \ann_R(W_{\nn-\ee_i}).$ Then 
\[
z_i \circ (f \circ H_\nn^j) = f \circ (z_i \circ H_\nn^j)=f \circ H_{\nn-\ee_i}^j=0.
\]
Hence
\[
f \circ H_\nn^j \in  W_\nn\cap V_\nn^i \subseteq W_{\nn-\ee_i}.  
\]
Therefore $f \circ H_\nn^j \in W_{\nn-\ee_i} \cap V_{\nn-\ee_i}^i $ as $z_i \circ (f \circ H_\nn^j)=0.$ Hence by \eqref{Eqn:Cond3} $ f \circ H_\nn^j \in W_{\nn-2\ee_i}. $ Repeating the same argument, we get 
$f \circ H_\nn^j \in W_{\nn-(n_i-1)\ee_i}.$ 
As $z_i^{n_i-1} \in \ann_R(W_{\nn-\ee_i})$ and $z_i^{n_i-1} \circ H_{\nn}^j=H_{\nn-(n_i-1)\ee_i}^j,$ 
the other containment in \eqref{Eqn:WeakCondition} follows. 

\vskip 3mm
Moreover, an $R$-submodule $W $ of $\D$ is $L_d^1$-admissible if and only if $W$ is $G_d$-admissible. 
Indeed,by the previous discussion it follows that if $W$ is $L_d^1$-admissible then $W$ satisfies \eqref{Eqn:WeakCondition} which is equivalent to \eqref{Eqn:GdAdmissible} as $\tau=1.$ Hence $W$ is $G_d$-admissible. Now, suppose that $W=\langle H_\nn:\nn \in \N^d_+ \rangle$ is $G_d$-admissible. Let $W_{\nn}:=\langle H_\nn\rangle.$ 
We claim that for all $\nn \in \N^d_+$ and $i=1,\ldots,d,$  
\begin{equation}
\label{Eqn:RossiCondition}
W_\nn \cap K[Z_1,\ldots,\widehat{Z_i},\ldots,Z_m] \subseteq W_{\nn-(n_i-1)\ee_i}.
\end{equation}
Consider $f \circ H_\nn \in W_{\nn} \cap K[Z_1,\ldots,\widehat{Z_i},\ldots,Z_m].$ Then $z_i \circ (f \circ H_\nn)=0.$ Hence 
\begin{eqnarray*}
 f \circ H_{\nn-\ee_i}=f \circ (z_i \circ H_{\nn})=z_i \circ (f \circ H_\nn)=0.
\end{eqnarray*}
This implies that $f \in \ann_R\langle H_{\nn-\ee_i} \rangle.$ Therefore by \eqref{Eqn:GdAdmissible} $f \circ H_\nn \in \langle H_{\nn-(n_i-1)\ee_i} \rangle=W_{\nn-(n_i-1)\ee_i}.$ 
To prove that $W$ is $L_d^1$-admissible it is enough to prove \eqref{Eqn:Cond3}. 
We prove \eqref{Eqn:Cond3} by induction on $n_i.$ 
Suppose $n_i=2$ and $f \circ H_{\nn} \in W_{\nn} \cap V_{\nn}^i.$ As $f \circ H_{\nn} \in V_{\nn}^i, $ $z_i \circ (f \circ H_{\nn})=0.$ Hence
\begin{eqnarray*}
 f \circ H_{\nn-\ee_i}=f \circ (z_i \circ H_{\nn})= z_i \circ (f \circ H_{\nn})=0. 
\end{eqnarray*}
Therefore $f \in \ann_R(W_{\nn-\ee_i})$ and hence by \eqref{Eqn:GdAdmissible} $f \circ H_{\nn} \in W_{\nn-\ee_i}.$

Now, assume that \eqref{Eqn:Cond3} is true for $n_i.$ 
Let $f \circ H_{\nn+\ee_i} \in W_{\nn+\ee_i} \cap V_{\nn+\ee_i}^i.$ Then 
\begin{eqnarray*}
z_i^{n_i-1} \circ (z_i \circ (f \circ H_{\nn+\ee_i}))=
z_i^{n_i} \circ (f \circ H_{\nn+\ee_i})=0. \hspace*{1cm}
\end{eqnarray*}
Also,
\[
z_i \circ (f \circ H_{\nn+\ee_i})=f \circ (z_i \circ H_{\nn+\ee_i})=f \circ H_\nn.
\]
Hence $ z_i \circ (f \circ H_{\nn+\ee_i}) \in W_\nn\cap V_\nn^i \subseteq W_{\nn-\ee_i}$ by induction. 
Thus there exists $g \in R$ such that
\[ 
z_i \circ (f \circ H_{\nn+\ee_i})=g \circ H_{\nn-\ee_i}=g\circ (z_i \circ H_\nn)= z_i \circ (g \circ H_{\nn}).
\]
This gives $z_i \circ (f \circ H_{\nn+\ee_i} - g \circ H_{\nn})=0$ and hence 
\[
f \circ H_{\nn+\ee_i} - g \circ H_{\nn} = (f-gz_i) \circ H_{\nn+\ee_i}\in W_{\nn+\ee_i} \cap K[Z_1,\ldots,\widehat{Z_i},\ldots,Z_m] \subseteq W_{(\nn+\ee_i)-n_i\ee_i}
\] where the last containment follows from \eqref{Eqn:RossiCondition}. As $W_{(\nn+\ee_i)-n_i\ee_i} \subseteq W_\nn,$ we get that $f \circ H_{\nn+\ee_i} \subseteq W_\nn$ as required.\qedhere
\end{discussion}

\section{Examples and Effective Constructions}\label{Section:Examples}

In this section we give few examples that illustrate Theorem \ref{thm:CharOfLocalLevel}. 
First we construct level $K$-algebras by constructing $L_d^\tau$-admissible submodule of $\D.$ 
In principle, we need an infinite number of polynomials to construct an $L_d^\tau$-admissible submodule. However, in the graded case we give an effective method to construct level $K$-algebra starting with finite $L_d^\tau$-admissible set (Proposition \ref{Prop:EffectiveConstruction}). We give an example to illustrate that if $d>0,$ then intersection of Gorenstein ideals need not be level (Example \ref{Example:intersection}), which is true in the Artinian case if Gorenstein ideals have same socle degree. Then we compute the inverse system of level $K$-algebras corresponding to certain semigroup rings (Example \ref{Example:semigroupring}) and Stanley-Reisner rings (Example \ref{Example:matroid}).

Following is a tautological example. Recall that an ideal $I \subseteq R =K[\![x_1,\ldots,x_m]\!]$ is a cone with respect to ideal $J \subseteq  K[\![x_{d+1},\ldots,x_m]\!]$ if $I=JR.$

\begin{proposition}\label{Prop:cone}
 Let $H^1,\ldots,H^\tau \in K_{DP}[X_{d+1},\ldots,X_m]$ be (resp. homogeneous) polynomials of same degree, say $b$, whose forms of degree $b$ are linearly independent. Consider the following $R$-submodule of $\D$
 \[
  W:=\langle H_\nn^j=X_1^{n_1-1}\ldots X_d^{n_d-1}H^j:j=1,\ldots,\tau,\nn \in \N^d_+\rangle.
 \]
Then $R/\ann_R(W)$ is a $d$-dimensional (resp. homogeneous) level $K$-algebra of type $\tau.$ Moreover, $\ann_R(W)$ is a cone with respect to $\ann_S\langle H^1,\ldots,H^\tau\rangle$ where $S=K[\![x_{d+1},\ldots, x_m]\!]$ (resp. $S=K[x_{d+1},\ldots,m]$). 
\end{proposition}
\begin{proof}
 First we show that $W$ is $L_d^\tau$-admissible $R$-submodule of $\D$ with respect to the sequence $\ux:=x_1,\ldots,x_d.$ Notice that $\ux$ is a sequence of general elements in $\mathcal{M}.$ It is clear that for each $\nn \in \N^d_+,$ 
 \[
 \deg H_\nn^1=\cdots=\deg H_{\nn}^\tau=b+|\nn|-d
 \] 
 and the forms of degree $b+|\nn|-d$ of $H_{\nn}^1,\ldots,H_{\nn}^\tau$ are linearly independent. Also, for each $j=1,\ldots,\tau,$ $i=1,\ldots,d$ and $\nn \in \N^d_+$ 
 \[
  x_i \circ H_{\nn}^j=\begin{cases}
                       X_1^{n_1-1}\cdots X_i^{n_i-2} \cdots X_d^{n_d-1} H^j=H_{\nn-\ee_i}^j & \mbox{ if }\nn-\ee_i > \mathbf{0}_d\\
                       0 & \mbox{ otherwise }
                      \end{cases}
 \]
and hence $W$ satisfies \eqref{Eqn:Cond2}. Next  
we prove \eqref{Eqn:Cond3} in the Definition \ref{Def:LdAdmissible}. Fix $1 \leq i \leq d.$ 
Let $W_\nn=\langle H_{\nn}^j:j=1,\ldots,\tau\rangle.$ Suppose that $\nn-\ee_i > {\bf 0}_d$ and $f \circ H_{\nn}^j \in W_{\nn} \cap V_\nn^i. $
As $f \circ H_{\nn}^j \in V_{\nn}^i,$ $x_i^{n_i-1} \circ (f \circ H_{\nn}^j)=0.$ This implies that $x_i$ divides $f.$ Let $f=x_ig$ for some $g \in R.$ 
Then 
\[    
f \circ H_{\nn}^j=(x_ig) \circ H_{\nn}^j=g\circ (x_i \circ H_{\nn}^j)=g \circ H_{\nn-\ee_i}^j \in W_{\nn-\ee_i}.
\]
Thus $W_{\nn} \cap V_\nn^i \subseteq W_{\nn-\ee_i}$ for all $\nn-\ee_i > {\bf 0}_d.$ Hence $W$ is $L_d^\tau$-admissible. Therefore by Proposition \ref{Prop:AdmissibleGivesLevel} $R/\ann_R(W)$ is a level $K$-algebra of dimension $d$ type $\tau.$ Hence we have the first part of the statement.

\noindent
It is easy to verify that
\[
 \ann_R(W_\nn)=(\ux^\nn)+(\ann_S\langle H^1,\ldots,H^\tau \rangle)R
 =(\ux^\nn)+(\ann_S(W_{{\bf 1}_d}))R.
\]
Therefore
\begin{eqnarray*}
  \ann_R(W)&=&\cap_{\nn \in \N_d^+}\ann_R(W_\nn)\\
&=&\cap_{\nn \in \N_d^+} (\ux^\nn)+(\ann_S(W_{{\bf 1}_d}))R\\
&=&(\ann_S(W_{{\bf 1}_d}))R \hspace*{1cm}
 \end{eqnarray*}
where the last equality follows from Krull intersection theorem applied to $R/\ann_S(W_{{\bf 1}_d})$.
Hence $\ann_R(W)$ is a cone with respect to $\ann_S(W_{{\bf 1}_d}).$ \qedhere
\end{proof}

Let $t_0 \in \N_+.$ We say that a family $\mathcal{H}=\{H_\nn^j:j=1,\ldots,\tau,~\nn \in \N^d_+, ~|\nn| \leq t_0\}$ of polynomials of $\D$ is $L_d^\tau$-admissible if the elements $H_\nn^j$ satisfy conditions \ref{Def:LdAdmissible}\eqref{Def:Cond1}, \ref{Def:LdAdmissible}\eqref{Def:Cond2} and \ref{Def:LdAdmissible}\eqref{Def:Cond3} of the definition of $L_d^\tau$-admissibility up to $\nn$ such that $|\nn| \leq t_0.$  
A similar proof as \cite[Proposition 4.2]{ER17} shows that in the graded case finitely many admissible polynomials $\mathcal{H}$ are suffices to get an ideal $I$ such that $R/I$ is level.

\begin{proposition}
\label{Prop:EffectiveConstruction}
 Let $\mathcal{H}=\{H_\nn^j:j=1,\ldots,\tau,~\nn \in \N^d_+, ~|\nn| \leq t_0\}$ be an $L_d^\tau$-admissible set of homogeneous polynomials with respect to a regular linear sequence $\uz=z_1,\ldots,z_d$ with $t_0 \geq (s+2)d$ where $s=\deg H_{{\bf 1}_d}^1.$ 
 Suppose that $\mathcal{H}$ can be extended to an $L_d^\tau$-admissible submodule $W=\langle G_{\nn}^j:\nn\in \N^d_+,j=1,\ldots,\tau \rangle$ of $\D$ such that $G^j_{\nn}=H^j_{\nn}$ for all $j=1,\dots,\tau$ and $|\nn|\le t_0.$ Let $I:=\ann_R(W).$ Then
 \[
  I=(\ann_R\langle H_{({\bf s}+{\bf 2})_d}^j:j=1,\ldots,\tau \rangle)_{\leq {s+1}} R.
 \]
\end{proposition}
\begin{proof}
 By Theorem \ref{thm:ExtraProperties} $R/I$ is a level $K$-algebra and 
 \[
   \reg(R/I)=\deg H_{{\bf 1}_d}^1=s.
 \]
 It is well known that the maximum degree of a minimal system of generators of $I$ is at most $\reg(R/I)+1.$ By using the Claim 1 in the proof of Proposition \ref{Prop:AdmissibleGivesLevel} we get 
 \[
 \ann_R \langle H_{({\bf s}+{\bf 2})_d}^j:j=1,\ldots,\tau \rangle= \ann_R(W_{({\bf s}+{\bf 2})_d})=I+(\uz^{({\bf s}+{\bf 2})_d}).
 \] 
 This gives the result.
\end{proof}

\begin{remark}
Notice that Proposition \ref{Prop:EffectiveConstruction} provides an effective method to construct an level graded $K$-algebra starting with a finite $L_d^\tau$-admissible set. Indeed, suppose that 
$\mathcal{H}$ is a finite $L_d^\tau$-admissible set as in Proposition \ref{Prop:EffectiveConstruction}. 
Check whether the ideal $I:=\ann_R(W_{({\bf s}+{\bf 2})_d})_{\leq {s+1}}$ is level. If so, then we have already constructed a level ring. Else, by Proposition \ref{Prop:EffectiveConstruction} $\mathcal{H}$ is not extendable to an $L_d^\tau$-admissible submodule. 
\end{remark}

We give an explicit example to demonstrate Propositions \ref{Prop:cone} and  \ref{Prop:EffectiveConstruction}.

\begin{example}
\pushQED{\qed}
Let $R=\mathbb{Q}[x,y,z]$ and  $\D=\mathbb{Q}_{DP}[X,Y,Z].$  
Let $d=1,~\tau=2$ and
\begin{align*}
	H_1^1&=Y^3 &H_1^2&=Z^3\\
	H_2^1&=XH_1^1 & H_2^2&=XH_1^2\\
	H_3^1&=X^2H_1^1  &H_3^2&=X^2H_1^2\\
	H_4^1&=X^3H_1^1 & H_4^2&=X^3H_1^2\\
	H_5^1&=X^4H_1^1 & H_5^2&=X^4H_1^2.
	\end{align*}
	By Proposition \ref{Prop:cone} the set $\mathcal{H}=\{H_1^1,H_1^2,H_2^1,H_2^2,H_3^1,H_3^3,H_4^1,H_4^2\}$ is $L_1^2$-admissible.
	Suppose that $\mathcal{H}$ can be extended to an $L_1^2$-admissible submodule $W$ of $\D.$ Then by Proposition \ref{Prop:EffectiveConstruction} $R/I$ is level where $I:=(\ann_R\langle H^1_5,H^2_5\rangle)_{\le 4}.$	
	By computer algebra system it can be verified that 
	\[
	I=(y^4,yz,z^4)
	\]
	and $R/I$ is a $1$-dimensional level $K$-algebra of type $2$. 
	Thus we have constructed level $K$-algebra starting with finite $L_1^2$-admissible set. \qedhere
\end{example}

From Proposition \ref{Prop:Emsalem-Iarrobino} it follows that the finite intersection of Gorenstein Artinian  ideals of same socle degree is level where by a Gorenstein (resp. level) ideal $I$ we mean that $R/I$ is Gorenstein (resp. level). 
This is no longer true if ideals have positive codimension as the following example illustrates. 

\begin{example}
\label{Example:intersection}
\pushQED{\qed}
 Let $R=\mathbb{Q}[x,y,z]$ and  $\D=\mathbb{Q}_{DP}[X,Y,Z].$ Let 
 \begin{align*}
  H_1^1&=Y^3-Z^3 &H_1^2&=Y^2Z\\
  H_2^1&=XH_1^1+YZ^3 & H_2^2&=XH_1^2\\
  H_3^1&=XH_2^1-Y^2Z^3  &H_3^2&=XH_2^2\\
  H_4^1&=XH_3^1+Y^3Z^3-4Z^6 & H_4^2&=XH_3^2 \\
  H_5^1&=XH_4^1+Y^7-Y^4Z^3+4YZ^6 & H_5^2&=XH_4^2.
 \end{align*}
By using Singular or Macaulay2 one can verify that the set $\mathcal{H}_1=\{H_1^1,H_2^1,H_3^1,H_4^1,H_5^1\}$
is $G_1$-admissible and the ideal 
 \[
  I:=(\ann_R\langle H_5^1\rangle)_{\leq 4}=(yz+xz,y^3+z^3-xy^2+x^2y-x^3)
 \]
is a $1$-dimensional Gorenstein ideal 
(see \cite[Example 4.4]{ER17}). 
Similarly, the set $\mathcal{H}_2=\{H_1^2,H_2^2,H_3^2,H_4^2,H_5^2\}$ is $G_1$-admissible and the corresponding $1$-dimensional Gorenstein ideal is 
\[
J:=(\ann_R\langle H_5^2\rangle)_{\leq 4}=(z^2,y^3).
\]  In fact, both $I$ and $J$ are complete intersections. It is easy to check that
\[
 I \cap J = (xz^2+yz^2,4y^3z+z^4,x^3y^3-x^2y^4+xy^5-y^6-y^2z^3)
\]
and $R/(I \cap J)$ has a graded minimal $R$-free resolution as follows:
\[
 0 \to R(-6) \oplus R(-7) \to R(-3) \oplus R(-4) \oplus R(-6) \to R \to 0.  
\]
This shows that $R/(I \cap J)$ is not level. Here 
\[
 Z^3 \in \langle H_2^1,H_2^2 \rangle \cap V_2^1 \setminus \langle H_1^1, H_1^2\rangle
\]
and hence $W:=\{H_n^j:n \leq 2 \mbox{ and } j=1,2  \}$ does not satisfy \eqref{Eqn:Cond3}. However, 
as $\mathcal{H}_1$ and $\mathcal{H}_2$ are $G_1$-admissible, it is easy to verify that $W$ satisfies \eqref{Eqn:WeakCondition}. \qedhere 
\end{example}

We now construct a one-dimensional level ring of type $2$ in the non-graded case.   
\begin{example}
	\label{Example:semigroupring}
	\pushQED{\qed}
	Let $R=\mathbb{Q}[x,y,z,w]$ and  $\D=\mathbb{Q}_{DP}[X,Y,Z,W].$ Let $d=1, ~\tau=2$ and
	\begin{align*}
	H_1^1&=YW &H_1^2&=ZW\\
	H_2^1&=XH_1^1 & H_2^2&=XH_1^2+Y^2W\\
	H_3^1&=X^2H_1^1+Z^2W  &H_3^2&=X^2H_1^2+XY^2W=XH_2^2\\
	H_4^1&=X^3H_1^1+XZ^2W+Y^2ZW+W^3 & H_4^2&=X^3H_1^2+X^2Y^2W+YZ^2W\\
	&=XH_3^1+Y^2ZW+W^3 & & =XH_3^2+YZ^2W.
	\end{align*}
	By using Singular or Macaulay2 one can verify that the set $\mathcal{H}_1=\{H_1^1,H_1^2,H_2^1,H_2^2,H_3^1,H_3^3,H_4^1,H_4^2\}$ is $L_1^2$-admissible.
	The involved polynomials are not homogeneous, and hence they correspond to a local ring.
	Therefore we can't use Proposition \ref{Prop:EffectiveConstruction}. 
	However, 
	\[
	\ann_R(W_4)=\ann_R\langle H_4^1,H_4^2\rangle=(x^4,y^2-xz,x^3-yz,x^2y-z^2,w^2-x^3y)
	\]
	and can be written as $I+(x^4)$, where
	\[
	I=(y^2-xz,x^3-yz,x^2y-z^2,w^2-x^3y).
	\]
	Moreover, $R/I$ is a $1$-dimensional level ring of type $2$ (see Example \ref{Example:Rossi-Valla}).
	Notice that $R/I\cong k[\![t^6,t^8,t^{10},t^{13}]\!].$ \qedhere
\end{example}

Example \ref{Example:semigroupring} raises the natural question whether it is possible to find an analogous of Proposition \ref{Prop:EffectiveConstruction} in the local case, at least for some special classes.

We now give a couple of examples of inverse systems of special families of level algebras.
\begin{example} (Semigroup rings)
\pushQED{\qed}
Let $n_1,\dots,n_l$ be an arithmetic sequence, i.e. there exists an integer $q \geq 1$ such that
\[
n_i=n_{i-1}+q=n_1+(i-1)q
\] 
for $i=2,\dots,l$.
Then the ring $A=K[t^{n_1},\dots,t^{n_l}]$ is a semigroup ring whose associated graded ring $G$ is level (see \cite[Proposition 1.12]{MT95}).
By \cite[Example 1(b)]{Fro87} the type of  $G$ is always greater or equal than the type of $A$. 
If $G$ is level, then the two types coincide. Hence we can deduce that for any minimal reduction $J$ of $\m,$ $A/J$ is level. Therefore by Proposition \ref{Prop:construction of level} the local ring $A$ is also level.

We give an explicit example. Let $A=\mathbb{Q}[\![t^6,t^{10},t^{14},t^{18}]\!]$.
Then $A$ is a semigroup ring associated to an arithmetic sequence, and we know from the previous discussion that $A$ is level.
It is easy to check that $A=R/I$ where $R=\mathbb{Q}[\![x,y,z,w]\!]$ and 
\[
	I=(x^3-w,x^4-yz, xz-y^2,x^3y-z^2).
\] 
Then $(x)$ is a general minimal reduction of $\mathcal{M}.$ 
Consider the following polynomials in $\D=\mathbb{Q}_{DP}[X,Y,Z,W]$ corresponding to the inverse system of $ I+(x^n)$ for $n \leq 5$:
\begin{align*}
	H_1^1&=Y & H_1^2&=Z \\
	H_2^1&=XH_1^1 & H_2^2&=XH_1^2+Y^2 \\
	H_3^1&=X^2H_1^1  & H_3^2&=X^2H_1^2+XY^2=XH_2^2\\
	H_4^1&=X^3H_1^1+YW+Z^2 & H_4^2&=X^3H_1^2+X^2Y^2+ZW=XH_3^2+ZW \\
	H_5^1&=X^4H_1^1+XYW+XZ^2+Y^2Z
	& H_5^2&=X^4H_1^2+X^3Y^2+XZW+YZ^2+Y^2W \\
	&=XH_4^1+Y^2Z &&=XH_4^2+YZ^2+Y^2W
\end{align*}
In principle by Theorem \ref{thm:CharOfLocalLevel} we have an infinite number of polynomials in the inverse system.
However, in this case, to recover the ideal we only need a finite number of polynomials. 
Let $W_{(5,5)}=\langle H_5^1,H_5^2\rangle$.
Using Singular one can verify that
\begin{align*}
	&\ann_R(W_{(5,5)})_{\le 4}=(x^3-w,xz-y^2,xw-yz,z^2-yw,x^5 )_{\le 4}\\
	&=( x^3-w,xz-y^2,xw-yz,z^2-yw)=I. \qedhere
\end{align*}
\end{example}

\vskip 2mm
In the following example we construct the inverse system of a Stanley-Reisner ring (of dimension two) associated to a matroid simplicial complex. 

\begin{example}	(Stanley-Reisner rings)
	\label{Example:matroid}
	\pushQED{\qed}
	By \cite[Theorem 3.4]{Sta96} the Stanley-Reisner rings associated to matroid simplicial complexes are level. We describe particular type of matroids that arise from matrices.
	Let $X\in K^{m\times n}$ where $K$ is a field and $m\le n.$  
	We write $[i_1,\dots,i_m]$ for the $m\times m$ minor of $X$ corresponding to columns $i_1,\ldots,i_m$ of $X$ where $1\le i_1<\cdots<i_m\le n.$ Consider the simplicial complex $\Delta(X)$ on vertices $\{1,\dots,n\}$ that is generated by 
	faces $\{i_1,\dots,i_m\},$ $1\le i_1<\cdots<i_m\le n,$ such that $[i_1,\dots,i_m]\ne 0,$ that is,
	\[
	\Delta(X)=\langle \{i_1,\dots,i_m\}\mid [i_1,\dots,i_m]\ne 0 \rangle.
	\] 
	Then $\Delta(X)$ is a matroid.
	Stanley's result implies that $R/I_{\Delta(X)}$ is a graded level algebra, where $I_{\Delta(X)}$ is the Stanley-Reisner ideal associated to $\Delta(X)$. Let us consider an explicit example.  

Let 
	\[
	X=\left(\begin{matrix}
	1 & 0 & 2 & 0 & 3 \\
	0 & 1 & 0 & 2 & 0 
	\end{matrix}\right)
	\] 
	Then the matroid associated to $X$ is
\[
\Delta:=\Delta(X) =\langle \{1,2\},\{2,3\},\{3,4\},\{4,5\},\{1,4\},\{2,5\}\rangle.
\]
	The figure below illustrates this simplicial complex:
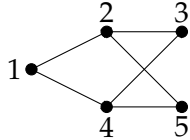
\begin{figure}[H]
	\centering	
	\begin{tikzpicture}[inner sep=1.5,scale=0.5]
	\draw[->] (0,1) -- (2,0);
	\draw[->] (0,1) -- (2,2);
	\draw[->] (4,0) -- (2,2);
	\draw[->] (2,0) -- (4,2);
	\draw[->] (4,2) -- (2,2);
	\draw[->] (2,0) -- (4,0);
	
	\draw (4,0) node[shape=circle,draw,fill=black] {};
	\draw (4,0) node[below=2pt] {$5$};
	\draw (4,2) node[shape=circle,draw,fill=black]{};
	\draw (4,2) node[above=2pt] {$3$};
	\draw (2,0) node[shape=circle,draw,fill=black]{};
	\draw (2,0) node[below=2pt] {$4$};
	\draw (0,1) node[shape=circle,draw,fill=black]{};
	\draw (0,1) node[left=2pt] {$1$};
	\draw (2,2) node[shape=circle,draw,fill=black]{};
	\draw (2,2) node[above=2pt] {$2$};
	\end{tikzpicture}
\end{figure}
\noindent 
Let $R=\mathbb{Q}[x_1,x_2,x_3,x_4,x_5].$ The Stanley-Reisner ideal corresponding to $\Delta$ is
	\[
	I_{\Delta}=(x_1x_3,x_2x_4,x_1x_5,x_3x_5).
	\]
	By Stanley's result, the ring $R/I_{\Delta}$ is a graded level ring of dimension $2$.
	Observe that $x_2+x_4$ and $x_1+x_3+x_5$ forms a regular sequence for $A=R/I_{\Delta}$.
	In order to find its inverse system, we first operate a change of coordinates:
	\begin{align*}
	\varphi: R & \to S=\mathbb{Q}[y_1,\dots,y_5] \\
	x_2+x_4 & \mapsto y_1 \\
	x_1+x_3+x_5 &\mapsto y_2 \\
	x_3 &\mapsto y_3 \\
	x_4 &\mapsto y_4 \\
	x_5 &\mapsto y_5
	\end{align*}
	Under this change of coordinates we get the ideal
	\[
	I:=\varphi(I_{\Delta})=((y_2-y_3-y_5)y_3,(y_1-y_4)y_4,(y_2-y_3-y_5)y_5, y_3y_5) \subseteq S.
	\]
	Now, the ring $A=S/I$ is a graded level ring of dimension $2$ and type $2$, and $y_1,y_2$ forms a regular sequence for $A$. Consider $\D=\mathbb{Q}_{DP}[Y_1,\dots,Y_5]$ the divided power ring dual to $S$.
	Using Singular or Macaulay2, we can compute the first generators of $I^\perp\subseteq \D$:
	\begin{align*}
	H_{(1,1)}^1=&Y_4Y_5 &H_{(1,1)}^2=&Y_3Y_4\\
	H_{(1,2)}^1=&Y_2H_{(1,1)}^1+Y_4Y_5^2& 	H_{(1,2)}^2=&Y_2H_{(1,1)}^2+Y_3^2Y_4\\
	H_{(2,2)}^1=&Y_1H_{(1,2)}^1+Y_2Y_4^2Y_5+Y_4^2Y_5^2 & H_{(2,2)}^2=&Y_1H_{(1,2)}^2+Y_2Y_3Y_4^2+Y_3^2Y_4^2 \\
	& \vdots & & \vdots \\
	H_{(4,4)}^1=&Y_1^2Y_2^2H_{(2,2)}^1 +Y_1Y_2^3Y_4^3Y_5+Y_2^3Y_4^4Y_5+ &H_{(4,4)}^2=&Y_1^2Y_2^2H_{(2,2)}^2+Y_1^3Y_2Y_3^3Y_4+Y_1^3Y_3^4Y_4+
	\\
	&Y_1Y_2^2Y_4^3Y_5^2+Y_2^2Y_4^4Y_5^2+Y_1^3Y_2Y_4Y_5^3+ & & Y_1^2Y_2Y_3^3Y_4^2+Y_1^2Y_3^4Y_4^2+Y_1Y_2^3Y_3Y_4^3+\\
	&Y_1^2Y_2Y_4^2Y_5^3+Y_1Y_2Y_4^3Y_5^3+Y_2Y_4^4Y_5^3+
	& & Y_1Y_2^2Y_3^2Y_4^3+Y_1Y_2Y_3^3Y_4^3+Y_1Y_3^4Y_4^3+ \\
	&Y_1^3Y_4Y_5^4+Y_1^2Y_4^2Y_5^4+Y_1Y_4^3Y_5^4+ & &Y_2^3Y_3Y_4^4+Y_2^2Y_3^2Y_4^4+Y_2Y_3^3Y_4^4+ \\
	& Y_4^4Y_5^4 &&Y_3^4Y_4^4.
	\end{align*}
	The set 
	\[
	\mathcal{H}=\{H_{(1,1)}^1,H_{(1,1)}^2,H_{(1,2)}^1,H_{(1,2)}^2,H_{(2,1)}^1,H_{(2,1)}^2,H_{(2,2)}^1,H_{(2,2)}^2,\dots,H_{(4,4)}^1,H_{(4,4)}^2\}
	\]
	is $L_2^2$-admissible. By Proposition \ref{Prop:EffectiveConstruction} we know that $I=\ann_R(W_{(4,4)})_{\le 3}$ where $W_{(4,4)}=\langle H_{(4,4)}^1,H_{(4,4)}^2 \rangle.$ In fact, it can be verified that
	\begin{align*}
	&\ann_R(W_{(4,4)})_{\le 3}=(y_3y_5, y_2y_5-y_5^2, y_1y_4-y_4^2, y_2y_3-y_3^2, y_2^4
	y_1^4, y_5^5, y_4^5, y_3^5)_{\le 3} \\
	=& (y_3y_5, y_2y_5-y_5^2, y_1y_4-y_4^2, y_2y_3-y_3^2)=I. \qedhere
	\end{align*}
\end{example}
\vskip 2mm

\section*{Acknowledgements}

We thank M. E. Rossi for suggesting the problem and providing many useful ideas throughout the preparation of this manuscript. We thank  Juan Elias for providing us the updated version of {\sc Inverse-syst.lib} and clarifying our doubt in Singular. We would also like to thank Alessandro De Stefani for providing us the proof of Proposition \ref{Prop:SocDegOfR_nn}\eqref{Prop:SocDegOfR_nnLocal} in the one-dimensional case, and Aldo Conca and Matteo Varbaro for useful discussions on the examples of level rings.



\end{document}